\documentclass[a4paper,12pt]{article}
\usepackage{latexsym,amssymb,amsfonts,amsmath,amsthm}
\usepackage[dvips]{graphicx}
\usepackage{color}
\newcommand{\segawa}{\textcolor{black}}
\newcommand{\HS}{\textcolor{black}}
\newcommand{\sato}{\textcolor{black}}

\newtheorem{theorem}{Theorem}
\newtheorem{lemma}{Lemma}
\newtheorem{corollary}{Corollary}
\newtheorem{proposition}{Proposition}
\newtheorem{remark}{Remark}
\newtheorem{definition}{Definition}

\newtheorem{assumption}{Assumption}

\newcommand{\bs}[1]{\boldsymbol{#1}}

\renewcommand{\labelenumi}{(\theenumi)}

\renewcommand{\labelenumi}{(\theenumi)}

\title{{\Large {\bf Spectral and asymptotic properties \\ of Grover walks on crystal lattices  
}
}}
\author{ 
{\small 
Yusuke Higuchi,$^{1}$ 
\footnote{higuchi@cas.showa-u.ac.jp 
}\quad 
Norio Konno,$^{2}$ 
\footnote{konno@ynu.ac.jp 
}\quad 
Iwao Sato,$^{3}$ 
\footnote{isato@oyama-ct.ac.jp 
}\quad  
Etsuo Segawa,$^{4}$ 
\footnote{Corresponding author: e-segawa@m.tohoku.ac.jp 
} 
}\\ 
{\scriptsize $^{1}$ 
Mathematics Laboratories, College of Arts and Sciences, Showa University
}\\
{\scriptsize 
Fuji-Yoshida, Yamanashi 403-005, Japan
}\\
{\scriptsize $^{2}$ 
Department of Applied Mathematics, Faculty of Engineering, Yokohama National University
}\\
{\scriptsize Hodogaya, Yokohama 240-8501, Japan
} \\
{\scriptsize $^3$ 
Oyama National College of Technology, 
}\\
{\scriptsize Oyama, Tochigi 323-0806, Japan
} \\
{\scriptsize $^4$ 
Graduate School of Information Sciences, Tohoku University
}\\
{\scriptsize Aoba, Sendai 980-8579, Japan
}\\ 
} 
\date{\empty }
\begin{document}
\maketitle

\par\noindent
\begin{small}
\par\noindent
{\bf Abstract}. 
We propose \segawa{a} twisted  Szegedy walk \segawa{for estimating the} limit behavior of a discrete-time quantum walk on 
\segawa{a crystal lattice, an infinite abelian covering graph, whose notion was introduced by \cite{KSS}}. 
\HS{First}, we show that the spectrum of the twisted Szegedy walk on the quotient graph \segawa{can be} expressed by mapping 
the spectrum of a twisted random walk \segawa{onto} the unit circle. 
Secondly, we show that the spatial Fourier transform of the twisted Szegedy walk \segawa{on a finite graph} with \segawa{appropriate} parameters 
becomes the Grover walk on its infinite abelian covering graph. 
Finally, as \segawa{an} application, we show that if the Betti number of the quotient graph is strictly greater than one, then localization is ensured 
with some appropriated initial state. 
\segawa{We also compute} the limit density function for the Grover walk on $\mathbb{Z}^d$ with flip flop shift, which implies 
the coexistence of linear spreading and localization. 
We partially obtain the abstractive shape \segawa{of the limit density function: the support} is within the $d$-dimensional sphere \segawa{of radius $1/\sqrt{d}$}, 
and \segawa{$2^d$ singular points reside on the sphere's surface.}

\end{small}

\setcounter{equation}{0}

\section{Introduction}
Quantum walks \segawa{have been} intensively studied from \segawa{various perspectives}. 
\segawa{A} primitive form \segawa{of a} discrete-time quantum walk on $\mathbb{Z}$ \segawa{can be found in the so called} Feynman's checker board \cite{FH}. 
\segawa{An actual} discrete-time quantum walk itself was \segawa{introduced} in \segawa{a study of quantum probabilistic theory \cite{Gudder}}. 
The Grover walk on general \segawa{graphs was proposed in \cite{Watrous}}. 
This \segawa{quantum walk has been} intensively-investigated \segawa{in} quantum information theory 
and spectral graph theory \cite{Am,HKSS1}, \segawa{and is known to accomplish a quantum speed up search in certain cases.} 
A review on applications of quantum walks to the quantum search algorithms is presented in \cite{Am}. 
\segawa{To enable more abstract interpretation of these quantum search algorithms, the Grover walk was generalized to the Szegedy walk in 2004~\cite{Sze}.} 
One \segawa{advance} of the Szegedy walk is that 
the performance of the quantum search algorithm based on the Szegedy walk is usually evaluated by hitting time of the \segawa{underlying} random walk. 

In this paper, we introduce \segawa{a} ``twisted" Szegedy walk \segawa{on} $\ell^2(D)$ for a given graph $G=(V,D)$, 
where $V$ and $D$ are the \segawa{sets} of vertices and arcs, respectively.  
In parallel, we also present a twisted (random) walk on $\ell^2(V)$ \segawa{underlying} the quantum walk on $\ell^2(D)$ \segawa{for providing a spectral mapping theory}. 
\segawa{Study on this twisted random walks has been well developed; the effects of some spatial structures of the crystal lattice on 
the return probability and the central limit theorems of the random walks were clarified~\cite{KS,KSS}, for example.} 
To elucidate the relationship between the spectra of the twisted random walk and the twisted Szegedy walk, 
we introduce new boundary operators $d_A,d_B: \ell^2(D)\to \ell^2(V)$. 
Then we show that \HS{$\mathcal{L}\equiv d_A^*(\ell^2(V))+d_B^*(\ell^2(V))\subset \ell^2(D)$} 
is invariant under the action of the twisted Szegedy walk. 
Here $d_A^*,d_B^*: \ell^2(V)\to \ell^2(D)$ are the adjoint operators of $d_A$ and $d_B$, respectively. 
\segawa{Careful observation reveals} that the eigenvalues of the twisted random walk on $\ell^2(V)$ describe the ``real parts" 
of \segawa{the} eigenvalues of the quantum walk (Proposition~\ref{key}). 
Thus we call the eigenspace $\mathcal{L}$ \segawa{the} inherited part from the twisted random walk. 
\segawa{The remainder} of the eigenspace; $\mathcal{L}^{\bot}$, 
is expressed by the intersection of the kernels of the boundary operators $d_A$ and $d_B$. 
\segawa{For a typical boundary operator of graphs $\partial:C_1\to C_0$, $\mathrm{ker}(\partial)$ is generated by all closed paths of $G$. 
Here $C_1=\sum_{e\in D(G)}\mathbb{Z}\delta_e$ and $C_0=\sum_{v\in V(G)}\mathbb{Z}\delta_v$. 
We show that the orthogonal complement space 
$\mathcal{L}^{\bot}=\mathrm{ker}(d_A)\cap \mathrm{ker}(d_B)$ is also characterized by all closed paths of $G$ (Theorem~\ref{cycles}). }
This fact implies that there is a homological abstraction within the Grover walk, 
and \segawa{this abstraction is crucial} to provide a typical stochastic behavior called localization (see Theorem~\ref{loc_cycle}). 

There are many types of mapping theorems. For example, \segawa{Higuchi and Shirai 
provided how the spectra of the Laplacian changes under graph-operations in \cite{HS}.} 
\segawa{They mapped the spectrum of the Laplacian of the original graph $G$ 
to those of the line graph $LG$, the subdivision graph $SG$ and the para-line graph $PG$ of the original graph $G$. }
\segawa{Our work proposes an alternative mapping theorem for discrete-time quantum walks.}
\segawa{The twist is equivalent to the vector potential in the context of the quantum graph in \cite{HKSS2}. 
We have partially succeeded in finding a spectral mapping theorem and relating the discrete-time quantum walk 
to the quantum graph of a finite regular covering graph in \cite{HKSS2}. }
A spectral result on quantum walks on a graph with two infinite half lines \segawa{has been} also obtained in \cite{Fel_Hil}
from the view point of a scattering quantum theory. 
\segawa{Clarifying the connections between these related topics~\cite{Fel_Hil,HKSS2,HS} and this presented work is an interesting future problem.} 

\segawa{As an application of our proposed mapping theorem, we show that the Grover walk on a crystal lattice leads to not only localization but also linear spreading in some cases. }
We find that the absolutely continuous part of the spectrum of the random walk on the crystal lattice 
\segawa{is responsible for} linear spreading of the quantum walk (Theorem~\ref{linear_spreading}). 

\segawa{The limit distribution of discrete-time quantum walks on \HS{the} $1$-dimensional square lattice 
is the well-known Konno density function $f_K$~\cite{Konno,Konno2}}. 
\segawa{On \HS{the} $2$-dimensional square lattice, 
the limit distribution is explicitly obtained in special cases of quantum coins i.g., \cite{WK3}. } 
\segawa{On} $d$-dimensional square lattice with $d\geq 3$, although the asymptotics of the return probability at the origin with the moving shift is discussed in \cite{SKJ}, 
almost all of the limit distributions have \segawa{yet to be clarified}. 
Here we obtain the partial shape of the limit distributions of the Grover walk on $\mathbb{Z}^d$: 
\segawa{namely, we show that the support of the density function is contained by the $d$-dimensional ball of radius $1/\sqrt{d}$ and 
whose surface holds $2^d$ singular points (Theorems~\ref{ball} and \ref{singular}).} 

This paper is organized as follows. 
In Sect.~2, we propose the twisted Szegedy walk and provide the \segawa{spectral mapping theory} of the twisted Szegedy walk from a twisted random walk. 
Section 3 is devoted to the Grover walk on \segawa{the} crystal lattice and its Fourier transform. 
\segawa{In Sect.~4, we illustrate our concepts by concrete computations for triangular, hexagonal and $d$-dimensional square lattices to show exhibiting both linear spreading and localization on each of them. 
Finally we compute the limit density function of the Grover walk on $\mathbb{Z}^d$.} 
\section{Twisted Szegedy walks on graphs}
\subsection{\segawa{Definition of the twisted Szegedy walks on graphs}}
\segawa{First, we explain the graph construction.} 
Let $G = (V(G),E(G))$ be a connected graph \segawa{($G$ may have multiple edges and self-loops)} with a set $V=V(G)$ of vertices and a set $E=E(G)$ of unoriented edges. 
The set of arcs $D=D(G)$ is naturally extracted from $E(G)$ as follows:
the origin and the terminal vertices of an arc $e\in D(G)$ are denoted by $o(e)$ and $t(e)$, respectively.  
The inverse arc of $e$ is denoted by $\bar{e}$. So $e\in D(G)$ \HS{if and only if} $o(e)t(e)\in E(G)$. 
For a vertex $v$ of $G$, $\mathrm{deg}(v)$ stands for degree of $v$. 
A path $p$ is a sequence of arcs $(e_1,e_2,\dots,e_n)$ with $t(e_j)=o(e_{j+1})$ for any $j\in \{1,\dots,\segawa{n-1}\}$. 
We denote the origin and terminus of a path $p$ as $o(p)$, and $t(p)$, respectively. 
If $\{o(e_j)\}_{j=1}^n$ are distinct, then the path $p$ is called a simple path. 
We define the \segawa{sets} of all paths and simple ones of graph $G$ by $P(G)$ and $P_0(G)$, \segawa{respectively}. 
If $t(e_{n})=o(e_1)$ for $p=(e_1,e_2,\dots,e_n)\in P_0(G)$ holds, then $p$ is called \segawa{an} essential cycle. 

\segawa{For a countable set $\Gamma=\{\gamma_1,\gamma_2,\dots\}$, we denote by $\ell^2(\Gamma)$ the square summable Hilbert space generated by $\Gamma$ whose inner product is defined by 
\[\langle \psi, \phi \rangle=\sum_{j} \overline{\psi(\gamma_j)} \phi(\gamma_j),\] 
where $\overline{c}$ ($c\in \mathbb{C}$) is the complex conjugate of $c$.}
We \segawa{take} the standard basis of $\ell^2(D(G))$ as $\delta_{e}$, ($e\in D(G)$), such that 
\HS{$\delta_{e}(f)=1$ if $e=f$ and $\delta_{e}(f)=0$ otherwise. }
At each vertex $u\in V(G)$, we define \segawa{the} subspace $\mathcal{H}_u\subset \ell^2(D(G))$ by
	\[\mathcal{H}_u=\mathrm{span}\{\delta_{e}; o(e)=u\}.\] 
Thus $\ell^2(D(G))=\bigoplus_{u\in V(G)}\mathcal{H}_u$ holds. 
Define $w: D(G)\to \mathbb{C}$ called \segawa{a} {\it weight} such that $w(e)\neq 0$ for all $e\in D(G)$ and  
	\[ \sum_{e:o(e)=u}|w(e)|^2=1 \mathrm{\;for\; all\;} u\in V(G).\]
\segawa{A function $\theta: D(G)\to \mathbb{R}$ is called \segawa{a} {\it $1$-form} if}
	\[\theta(\bar{e})=-\theta(e) \mathrm{\;\;for\; any\;} e\in D(G). \]
To define the generalized Szegedy walk, we introduce two kinds of (weighted) boundary operators: 
\segawa{
\begin{definition}
\HS{Let $d_A$, $d_B$ be boundary operators for $\phi\in \ell^2(D(G))$ and $v\in V(G)$ such that }
\begin{align}\label{dualbounded}
	(d_A\phi)(v)=\sum_{e:o(e)=v}\phi(e)\overline{w(e)},\;\;\;(d_B\phi)(v)=\sum_{e:o(e)=v}\phi(\bar{e})\overline{w(e)}e^{-i\theta(e)},	
\end{align}
respectively. \HS{Here the symbol $i$ stands for the imaginary unit. }
\end{definition}
}
The coboundary operators of $d_A^*$ and $d_B^*$ are determined by for $f\in \ell^2(V(G))$ and $e\in D(G)$, 
\begin{align}\label{bounded}
	(d_A^*f)(e)=w(e)f(o(e)),\;\;\;(d_B^*f)(e)=e^{-i\theta(e)}w(\bar{e})f(t(e)), 
\end{align}
respectively from the relationship $\langle \phi,d_J^*f \rangle_D=\langle d_J\phi,f \rangle_V$ ($J\in\{A,B\}$) 
with $f\in \ell^2(V(G))$ and $\phi\in \ell^2(D(G))$. 
Here for any $\phi_1,\phi_2\in \ell^2(D(G))$ and $f_1,f_2\in \ell^2(V(G))$, 
\[ \langle \phi_1,\phi_2\rangle_D=\sum_{e\in D(G)}\overline{\phi_1(e)}\phi_2(e),\;\langle f_1,f_2\rangle_V=\sum_{v\in V(G)}\overline{f_1(v)}f_2(v). \]
\HS{\begin{lemma}\label{projection}
It holds that $d_Ad_A^*=d_Bd_B^*=I_V$. 
On the other hand, the operators $d_A^*d_A$ and $d_B^*d_B$ are the projections onto the subspaces 
\begin{align} 
	\mathcal{A} \equiv d_A^*\left(\;\ell^2(V(G))\;\right)\subset \ell^2(D(G)) \;\;\mathrm{and}\;\;
        \mathcal{B} \equiv d_B^*\left(\;\ell^2(V(G))\;\right)\subset \ell^2(D(G)),
\end{align}
respectively. 
So we have \[ d_A(\;\ell^2(D(G))\;)=d_B(\;\ell^2(D(G))\;)=\ell^2(V(G)).  \]
\end{lemma}}
%
\segawa{
\begin{definition}
The twisted Szegedy walk associated with the weight $w$ and the $1$-form $\theta$ is defined as follows: 
\begin{enumerate}
\item The total state space $\mathcal{H}$ : $\mathcal{H}=\ell^2(D(G))$. 
\item Time evolution $U^{(w,\theta)}$ : $U^{(w,\theta)}=S^{(\theta)}C^{(w)}$. Here 
\begin{enumerate}
\item The coin operator:  $C^{(w)}=2d_A^*d_A-I $
\item The (twisted) shift operator: for $f\in \ell^2(D(G))$, \[ (S^{(\theta)}f)(e)={\rm e}^{-{\rm i}\theta(e)}f(\bar{e}). \]
\end{enumerate}
\item The finding probability: let $\Psi_n\in \ell^2(D(G))$ be the $n$-th iteration of the walk such that $\Psi_n=U^{(w,\theta)}\Psi_{n-1}$ $(n\in \{1,2,\dots\})$ with $||\Psi_0||=1$. 
We define   $\nu_n: V(G)\to [0,1]$ such that 
	\[ \nu_n(u)=\sum_{e:o(e)=u}|\Psi_n(e)|^2. \]
$\nu_n(u)$ is called a finding probability of the twisted Szegedy walk at time $n$ at vertex $u$. 
\end{enumerate}
\end{definition}
}
%
\segawa{We remark that $\nu_n$ is a probability distribution by the unitarity of the time evolution $U^{(w,\theta)}$. } 
\HS{The motivation introducing the $1$-form $\theta$ to the quantum walk will be clear in Sect.~3. }
\HS{
\begin{remark}
On some regular graphs, there exists two possible choices of the shift operators: moving shift $S_m$ and flip-flop shift $S_f$. 
For one-dimensional square lattice, these shift operators are given by 
\begin{align*} 
	(S_m\psi)((x,x+j)) &= \psi((x-j,x)), \\ 
        (S_f\psi)((x,x+j)) &= \psi((x+j,x)),\;\; (j\in\{\pm 1\}, \;x\in\mathbb{Z}). 
\end{align*}
The moving shift maintains the direction of arcs while the flip-flop shift reverses them. 
However on a general graph, the direction, and hence the moving shift operator, is not uniquely determined. 
Moreover, any moving shift type quantum walks on unoriented graphs can be generated by permutating the local coin operator of a flip-flop shift type quantum walk. 
For one-dimensional square lattice, the time evolution of the flip-flop shift $U_f(H)$ is related to the moving shift $U_m(H)$ 
with local coin operator $H$ as follows: 
\[ U_m(H)=\mathcal{P}U_f(H\mathcal{P})\mathcal{P}. \]
Here $\mathcal{P}g((x,x+j))=g((x,x-j))$. 
More detailed and general discussions are given in Sect.~2 in \cite{HKSS1}. 
For these reasons, we apply the flip-flop shift operator in this paper. 
Discrete-time quantum walks with the moving shift on an {\it oriented} triangular lattice are implemented in \cite{KSKJ}. 
\end{remark}
}
%
%
\subsection{\segawa{Spectral map of the twisted Szegedy walk}}
\segawa{We compare the dynamics of the twisted Szegedy walk to the underlying random walk. }
To this end, we introduce $P(e): \mathcal{H}_{o(e)}\to \mathcal{H}_{t(e)}$ ($e\in D(G)$) by 
	\begin{equation}\label{weightP} P(e)=\Pi_{\delta_{\bar{e}}}U^{(w,\theta)}\Pi_{\mathcal{H}_{o(e)}}. \end{equation}
\segawa{Throughout this paper, we define $\Pi_{\mathcal{H}'}$ as the orthogonal projection onto $\mathcal{H}'$, \HS{a subspace $\mathcal{H'}$ of $\mathcal{H}$}. }
\sato{For $e\in D(G)$, we illustrate $P(e)$ in the following matrix expression: 
letting $\{e_1,\dots,e_{\kappa}\}$ and $\{f_1,\dots,f_{\kappa'}\}$ ($\kappa=\mathrm{deg}(o(e))$, $\kappa'=\mathrm{deg}(t(e))$) be all arcs with $o(e_j)=o(e)$ ($j\in\{1,\dots,\kappa\}$) and 
$o(f_j)=t(e)$ ($j\in\{1,\dots,\kappa'\}$), respectively, 
\[
P(e)=
\begin{array}{r|ccccc}
 & e & e_2 & \cdots & e_{\kappa} \\ \hline
\bar{e} & (2|w(e)|^2 -1)e^{i\theta(e)} & 2\overline{w(e_2)}w(e)e^{i\theta(e)} & \cdots & 2\overline{w(e_\kappa)}w(e)e^{i\theta(e)} \\
f_2     & 0 & 0 & \cdots & 0 \\
\vdots  & \vdots  & \vdots  &  \ddots    &  \vdots \\
f_{\kappa'} & 0 & 0 & \cdots & 0
\end{array},
\]
where we put $e_1=e$ and $f_1=\bar{e}$. }
Given an initial state $\Psi_0\in \ell^2(D(G))$ with $||\Psi_0||^2=1$, 
let $\Psi_n: V(G)\to \ell^2(D(G))$ be $\Psi_n(u)=\Pi_{\mathcal{H}_u}(U^{(w,\theta)})^n \Psi_0$ for $n\in \mathbb{N}$. 
Then we have the following recurrence relation: 
	\begin{equation}\label{analogue} 
	\Psi_n(u)=\sum_{e:t(e)=u}P(e)\Psi_{n-1}(o(e)),\;(n\geq 1), 
	\end{equation}
since 
\begin{align*}
	\Psi_n(u) &= \Pi_{\mathcal{H}_u}(U^{(w,\theta)})\sum_{v\in V(G)}\Psi_{n-1}(v)=\sum_{e:o(e)=u}\Pi_{\delta_e}(U^{(w,\theta)})\Psi_{n-1}(o(\bar{e})) \\
        	&= \sum_{e:o(e)=u}\Pi_{\delta_e}(U^{(w,\theta)})\Pi_{\mathcal{H}_{o(\bar{e})}}\Psi_{n-1}(o(\bar{e}))=\sum_{e:t(e)=u}P(e)\Psi_{n-1}(o(e)). 
\end{align*}
\segawa{This expression is analogous to the random walk on $G$ with transition probability $\{p(e):e\in D(G)\}$, where $p(e)\in (0,1]$ and $\sum_{e:o(e)=u}p(e)=1$; that is, 
we have $\{P(e):e\in D(G)\}$, where $P(e): \mathcal{H}_{o(e)}\to \mathcal{H}_{t(e)}$ and $\sum_{e: o(e)=u}P(e)=H^{(u)}\in \mathrm{SU}(\mathrm{deg}(u))$. }
\segawa{Now we introduce an operator which is essential to the rest of this paper. 
\begin{definition}\label{discriminant}
We define a self-adjoint operator on $\ell^2(V)$ as follows: 
	\[ T^{(w,\theta)}= d_A d_B^*. \]
We call $T^{(w,\theta)}$ a discriminant operator of $U^{(w,\theta)}$. 
\end{definition}
}
Applying Eqs.~(\ref{dualbounded}) and (\ref{bounded}), we have for $u,v\in \ell^2(V)$, 
\begin{equation}\label{T_uv}
	\langle \delta_v, T^{(w,\theta)}\delta_u \rangle=\sum_{f:\;o(f)=u,t(f)=v} w(f)\overline{w(\bar{f})}\;{\rm e}^{{\rm i}\theta(f)}.
\end{equation}
\segawa{When we specify $p(e)=|w(e)|^2$ $(e\in D(G))$ in the above random walk, we can relate this random walk to the discriminant operator $T^{(w,\theta)}$ through Proposition~2. 
Now we represent a spectral map from $\mathrm{spec}(T^{(w,\theta)})$ to $\mathrm{spec}(U^{(w,\theta)})$ as follows. 
}
%
\noindent \begin{proposition}\label{key}
\segawa{Assume that graph $G$ is finite; that is, $|V(G)|,|E(G)|<\infty$. Then we have the following spectral map. }
\begin{enumerate}
\item Eigenvalues: 
Denote by $m_{\pm 1}$ the multiplicities of the eigenvalues \segawa{$\pm 1$} of $T^{(w,\theta)}$, respectively. 
Let $\varphi_{QW}:\mathbb{R}\to \mathbb{R}$ such that $\varphi_{QW}(x)=(x+x^{-1})/2$. 
Then we have 
\begin{equation}
\mathrm{spec}(U^{(w,\theta)})=\varphi_{QW}^{-1}\left(\mathrm{spec}(T^{(w,\theta)})\right) \cup \{1\}^{M_1} \cup \{-1\}^{M_{-1}},
\end{equation}
where $M_1=\mathrm{max}\{0,|E|-|V|+m_1\}$, $M_{-1}=\mathrm{max}\{0,|E|-|V|+m_{-1}\}$. 
\item Eigenspace: 
The eigenspace of eigenvalues $\varphi_{QW}^{-1}\left(\mathrm{spec}(T^{(w,\theta)})\right)$, $\mathcal{L}$, is 
	\[ \HS{\mathcal{L}= d_A^*(\ell^2(V))+d_B^*(\ell^2(V)).} \]
The normalized eigenvector of the eigenvalue $e^{i\phi}\in \varphi_{QW}^{-1}\left(\mathrm{spec}(T^{(w,\theta)})\right)$ is given by 
	\segawa{
        \[ \begin{cases} d_A^*\nu_\phi & \text{: $\phi\in\{0,\pi\}$,} \\ \frac{1}{\sqrt{2}|\sin \phi|}(I-e^{i\phi}S^{(\theta)})d_A^*\nu_\phi, & \text{: $\phi\notin \{0,\pi\}$,} \end{cases}\]}
where $\nu_\phi=\nu_{-\phi}\in \ell^2(V)$ is the eigenfunction of $T^{(w,\theta)}$ of \segawa{the} eigenvalue $\cos\phi$ with $||\nu||=1$. 
The orthogonal complement space is given by 
\[ \mathcal{L}^\bot=\mathrm{ker}(d_A)\cap \mathrm{ker}(d_B).\]
The eigenspaces corresponding to eigenvalues $\{1\}^{M_1}$ and $\{-1\}^{M_{-1}}$ are described by
\begin{equation} \mathrm{ker}(d_A) \cap \mathcal{H}_-^{(S)}\mathrm{\;and\;} \mathrm{ker}(d_{\segawa{A}}) \cap \mathcal{H}_+^{(S)}, \end{equation}
respectively. 
Here 
	\[\mathcal{H}_{\pm}^{(S)}=\mathrm{span}\{ \delta_f\pm {\rm e}^{i\theta(f)} \delta_{\bar{f}}; f\in D(G)\}, \] 
in other words $\mathcal{H}_\pm^{(S)}$ are the eigenspaces of $S^{(\theta)}$ \segawa{with eigenvalues} $1$ and $-1$, respectively. 
\end{enumerate}
\end{proposition}
\begin{proof}
\segawa{First, we show that the absolute value of each eigenvalue of $T^{(w,\theta)}$ is bounded by $\pm 1$. }
When $T^{(w,\theta)}g=\lambda g$  for $g\in \ell^2(V)$, $\lambda\in \mathbb{R}$ because $T^{(w,\theta)}$ is a self adjoint operator. 
On the other hand, 
from the definition of $T^{(w,\theta)}$ and Remark~\ref{projection}, we can express $T^{(w,\theta)}$ in the alternative form:  
	\begin{equation}\label{T_another}
	T^{(w,\theta)}=d_A\Pi_\mathcal{A}S^{(\theta)}\Pi_\mathcal{A}d_A^*.
	\end{equation} 
\segawa{Two boundary operators $d_A$  and $d_B$ are related by 
	\[ d_B=d_AS^{(\theta)},\;\;d_B^*=S^{(\theta)}d_A^*. \] }
Equation~(\ref{T_another}) implies that 
	\begin{multline}\label{AcapB}
 	|\lambda|^2 ||g||^2= ||T^{(w,\theta)}g||^2 =|| \Pi_\mathcal{A}S^{(\theta)}\Pi_\mathcal{A}\psi ||^2 \\ 
 	\leq || S^{(\theta)}\Pi_\mathcal{A}\psi ||^2=|| \Pi_\mathcal{A}\psi ||^2 \leq ||\psi||^2=||g||^2, 
	\end{multline}
where $\psi=d_A^*g\in \ell^2(D)$. 
By Eq.~(\ref{T_another}), the second equation derives from 
$||T^{(w,\theta)}g||^2=\langle d_A\psi, T^{(w,\theta)}d_A\psi \rangle=\langle \psi, d_A^*T^{(w,\theta)}d_A\psi \rangle
=\langle \psi, \Pi_\mathcal{A}S^{(\theta)}\Pi_\mathcal{A}\psi \rangle$. 
Lemma~\ref{projection} implies $||\psi||^2=\langle d_A^*g, d_A^*g\rangle=\langle g, d_Ad_A^*g\rangle=||g||^2$. 
Thus  all the eigengenvalues of $T^{(w,\theta)}$ live in $[-1,1]$. 

On the other hand, noting $C^{(w)}d_A^*=d_A^*$ and $C^{(w)}d_B^*=2d_A^*T^{(w,\theta)}-d_B^*$, we have  
\begin{align}\label{U_dadb}
U^{(\theta,w)}d_A^* &= d_B^*, \\
U^{(\theta,w)}d_B^* &= 2d_B^*T^{(w,\theta)}-d_A^*.
\end{align}
\segawa{If $\nu_\phi\in \ell^2(V)$ is selected as the eigenfunction of $T^{(w,\theta)}$ of eigenvalue $\cos\phi$,} 
we observe that $\mathrm{span}\{d_A^*\nu_\phi,d_B^*\nu_\phi\}\subset \mathcal{L}$ is invariant under the action $U^{(w,\theta)}$ as follows: 
\begin{align}
U^{(\theta,w)}\rho_\phi &= S^{(\theta)}\rho_\phi, \label{A}\\
U^{(\theta,w)}S^{(\theta)}\rho_\phi &= 2\lambda S^{(\theta)}\rho_\phi-\rho_\phi, 
\end{align}
where $\rho_\phi=d_A^*\nu_\phi$. 
Define $\mathcal{L}_{\lambda} \equiv \mathrm{span}\{ \rho_\phi,S^{(\theta)}\rho_\phi \}$, where $\lambda=\cos\phi$. 
It holds that 
\begin{equation} 
\mathcal{L}=\bigoplus_{\lambda \in \mathrm{spec}(T^{(w,\theta)})} \mathcal{L}_{\lambda}. 
\end{equation}
Since 
\[|\langle \rho_\phi, S^{(\theta)}\rho_\phi \rangle|
=|\langle d_A^*\nu_\phi, S^{(\theta)}d_A^*\nu_\phi\rangle|=|\langle \nu_\phi,d_AS^{(\theta)}d_A^*\nu_\phi \rangle|
=|\langle \nu_\phi,T^{(w,\theta)}\nu_\phi \rangle|=\cos\phi, \]
the value $\phi\in \mathbb{R}$ is called \segawa{the} geometric angle between $\rho_\phi$ and $S^{(\theta)}\rho_\phi$. 
Thus \HS{ $\rho_\phi$ and $S^{(\theta)}\rho_\phi$ are linearly dependent if and only if $\phi\in\{n\pi;n\in \mathbb{N}\}$. }
\begin{enumerate}
\item $\phi\in\{n\pi;n\in \mathbb{N}\}$ case : In this case, since
$\langle \rho_\phi, S^{(\theta)}\rho_\phi \rangle\in \{\pm 1\}$ holds, then 
we have $S^{(\theta)}\rho_{0}=\rho_{0}$ and $S^{(\theta)}\rho_{\pi}=-\rho_{\pi}$. 
Thus 
\begin{equation}\label{1}
\mathcal{L}_1 \subset \mathcal{H}_+^{(S)} \mathrm{\;and\;} \mathcal{L}_{-1} \subset \mathcal{H}_-^{(S)}. 
\end{equation} 
From Eq.~(\ref{A}), 
\begin{equation} U^{(w,\theta)}\rho_0=\rho_0 \mathrm{\;and\;} U^{(w,\theta)}\rho_\pi=-\rho_\pi. \end{equation}
\item $\phi\notin\{n\pi;n\in \mathbb{N}\}$ case : In the subspace $\mathcal{L}_{\cos\phi} =\mathrm{span}\{ \rho_\phi,S^{(\theta)}\rho_\phi \}$, 
\[ \HS{U^{(w,\theta)}|_{\mathcal{L}_{\cos \phi}}} \cong  \begin{bmatrix} 0 &  -1 \\ 1 & 2\cos\phi \end{bmatrix}, \] 
where $\rho_\phi \cong {}^T[1,0]$ and $S^{(\theta)}\rho_\phi \cong {}^T[0,1]$. Then the eigenvalues and their normalized eigenvectors of eigenspace $\mathcal{L}_{\cos\phi}$ are given by 
\begin{equation} 
e^{\pm i\phi}\mathrm{\;and\;} \omega_{\pm \phi} = \frac{1}{\sqrt{2}|\sin \phi|} (I-e^{\pm i\phi}S^{(\theta)})\rho_{\pm \phi}, 
\end{equation}
respectively. 
Also noting that 
\[ S^{(\theta)}\omega_\phi=-e^{i\phi}\omega_{-\phi},\;\; S^{(\theta)}\omega_{-\phi}=-e^{-i\phi}\omega_{\phi},  \] 
we have 
\[ \omega_\phi + e^{i\phi}\omega_{-\phi} \in \mathcal{H}_-^{(S)},\;\; 
\omega_\phi- e^{i\phi}\omega_{-\phi} \in \mathcal{H}_+^{(S)}. \]
Therefore
\begin{equation}\label{2}
\mathrm{dim}(\mathcal{L}_{\cos\phi} \cap \mathcal{H}_+^{(S)})=\mathrm{dim}(\mathcal{L}_{\cos\phi} \cap \mathcal{H}_-^{(S)})=1. 
\end{equation}
\end{enumerate}
Combining Eqs.~(\ref{1}) and (\ref{2}), we obtain 
\begin{equation}\label{3}
\mathrm{dim}(\mathcal{H}_+^{(S)} \cap \mathcal{L}) =|V|-m_{-1} \mathrm{\;and\;} \mathrm{dim}(\mathcal{H}_-^{(S)} \cap \mathcal{L}) =|V|-m_{1}. 
\end{equation}
We now consider the orthogonal complement space of $\mathcal{L}$. 
By direct computation, 
we have 
$\psi\in \mathcal{H}_-^{(S)}$ and $\phi\in \mathcal{H}_+^{(S)}$ if and only if 
\begin{align}
\psi(e) &= -e^{-i\theta(e)}\psi(\bar{e}), \label{alpha}\\
\phi(e) &= e^{-i\theta(e)}\phi(\bar{e}), \label{alpha2}
\end{align}
respectively.
For any $\bs{\psi}\in \mathcal{L}^\bot$, it holds that 
\begin{align*}
\bs{\psi} & \stackrel{C^{(w)}}{\mapsto} -\bs{\psi} \stackrel{S^{(\theta)}}{\mapsto} -S^{(\theta)}\bs{\psi}, \\
S^{(\theta)}\bs{\psi} & \stackrel{C^{(w)}}{\mapsto} -S^{(\theta)}\bs{\psi} \stackrel{S^{(\theta)}}{\mapsto} -\bs{\psi}
\end{align*}
and therefore 
\begin{align}
U^{(w,\theta)}(\bs{\psi}-S^{(\theta)}\bs{\psi}) = (\bs{\psi}-S^{(\theta)}\bs{\psi}) \mathrm{\;and\;}
U^{(w,\theta)}(\bs{\psi}+S^{(\theta)}\bs{\psi}) = -(\bs{\psi}+S^{(\theta)}\bs{\psi}). 
\end{align}
Note that 
\begin{align*}
S^{(\theta)}(\bs{\psi}-S^{(\theta)}\bs{\psi}) &= -(\bs{\psi}-S^{(\theta)}\bs{\psi})\in \mathcal{H}_-^{(S)}, \\
S^{(\theta)}(\bs{\psi}+S^{(\theta)}\bs{\psi}) &= \bs{\psi}+S^{(\theta)}\bs{\psi} \in \mathcal{H}_+^{(S)}. 
\end{align*}
Since the dimension of the whole space of the DTQW is $|D|=2|E|$ and $\mathrm{dim}(\mathcal{H}_\pm^{(S)})=|E|$, we have, by Eq.~(\ref{3}), 
\[ \mathrm{dim}(\mathcal{H}_+^{(S)}\cap \mathcal{L}^{\bot})= |E|-(|V|-m_{-1}) \mathrm{\;and\;} \mathrm{dim}(\mathcal{H}_-^{(S)}\cap \mathcal{L}^{\bot})= |E|-(|V|-m_{1}). \]
So we obtain 
for any $\bs{\psi}\in \mathcal{H}_+^{(S)}\cap \mathcal{L}^{\bot}$, $U^{(w,\theta)}\bs{\psi}=-\bs{\psi}$ with the multiplicity $|E|-|V|+m_{-1}$, and 
for any $\bs{\psi}\in \mathcal{H}_-^{(S)}\cap \mathcal{L}^{\bot}$, $U^{(w,\theta)}\bs{\psi}=\bs{\psi}$ with the multiplicity $|E|-|V|+m_{1}$. 

On the other hand, we remark that if $\psi_\pm \in \mathcal{H}_\pm^{(S)}\cap \mathrm{ker}(d_A)$, then 
\[ d_B\psi_{\pm}(u)=d_AS^{(\theta)}\psi_{\pm}(u)=\pm d_A\psi_{\pm}(u)=0,\;\;(u\in V(G)). \]
Thus $\psi_\pm$ belongs also to $\mathrm{ker}(d_B)$. 
Moreover, 
$\mathcal{L}=\mathrm{span} \{\mathcal{A},\mathcal{B} \}$ implies that 
\[ \mathcal{L}^\bot=\mathrm{ker}(d_A)\cap \mathrm{ker}(d_B). \]
Therefore $\mathcal{H}_\pm^{(S)}\cap \mathrm{ker}(d_A)
	\subseteq \mathrm{ker}(d_A)\cap \mathrm{ker}(d_B)\cap \mathcal{H}^{(S)}_\pm
        =\mathcal{L}^\bot \cap \mathcal{H}^{(S)}_\pm$. 
So \segawa{putting  $\mathcal{M}_\pm$ as the eigenspaces of eigenvalues $\pm 1$, which are orthogonal to $\mathcal{L}$}, we can conclude that
\begin{equation}\label{beta}
\mathcal{M}_\pm=\mathrm{ker}(d_A) \cap \sato{\mathcal{H}_{\mp}^{(S)}}, 
\end{equation}
respectively. 
The proof is completed. 
\end{proof}
Let $G$ be a connected graph and $p:D(G)\to (0,1]$ a transition probability; that is, $\sum_{e:o(e)=u}p(e)=1$ for all $u\in V(G)$. 
If there exists a positive valued function $m:V(G)\to (0,\infty)$ such that 
\[ m(o(e))p(e)=m(t(e))p(\bar{e}), \]
for every $e\in D(G)$, then $p$ is said to be {\it reversible}, and $m$ is called \segawa{a} {\it reversible measure}. 
\begin{proposition}\label{gururi}
A random walk on $G$ with transition probability $p: D(G)\to [0,1]$ given by $p(e)=|w(e)|^2$, has a reversible measure 
if and only if there exists \segawa{a} $1$-form $\theta$ such that $1\in \mathrm{spec}(T^{(w,\theta)})$. 
\end{proposition}
\begin{proof}
We first show that \segawa{a} $1$-form $\theta$ exists such that $1\in \mathrm{spec}(T^{(w,\theta)})$ if and only if 
\begin{equation}\label{DBC}
m(o(e))\widetilde{w}(e)=m(t(e))\widetilde{w}(\bar{e}), 
\end{equation}
\segawa{where} $\widetilde{w}(e)= w(e)e^{i\theta(e)/2}$.
\segawa{Equation~(\ref{DBC}) is an extended detailed balanced condition (eDBC). }
The multiplicities \segawa{$m_{\pm 1}$ are} positive if and only if \segawa{both} equalities in Eq.~(\ref{AcapB}) hold. 
Note that the first equality in Eq.~(\ref{AcapB}) holds if and only if 
\[ || \Pi_\mathcal{A}S^{(\theta)}\Pi_\mathcal{A}\psi ||^2=|| S^{(\theta)}\Pi_\mathcal{A}\psi ||^2 
	\Leftrightarrow S^{(\theta)}\Pi_\mathcal{A}\psi\in \mathcal{A} 
        \Leftrightarrow \Pi_\mathcal{A}\psi\in \mathcal{B}. \]
On the other hand, the second equality in Eq.~(\ref{AcapB}) holds if and only if 
\[ || \Pi_\mathcal{A}\psi||^2 = ||\segawa{\psi}||^2
	\Leftrightarrow \psi\in \mathcal{A}. \]
Combining these, we find that $|\lambda|=1$ if and only if $\psi\in \mathcal{A}\cap \mathcal{B}$. 
Since $\psi\in \mathcal{A}$, \segawa{specifying an} eigenfunction \segawa{$m\in \ell^2(V)$} of $T^{(w,\theta)}$ for \segawa{the} eigenvalue $\lambda=1$ or $\lambda=-1$, we can write $\psi=d_A^*m$. 
On the other hand, since $\psi\in \mathcal B$, there exists $h\in \ell^2(V)$ such that 
\[d_B^*h=d_A^*m.\] 
\HS{Letting $d_A$ operate to both sides, we have $h=\pm m$ by Lemma~\ref{projection} and Definition~\ref{discriminant}.} 
Thus $\lambda=\pm 1$ if and only if 
\begin{equation}\label{1227}
d_A^*m=\pm d_B^*m. 
\end{equation}
Inserting the definitions of $d_A^*$ and $d_B^*$ in Eq.~(\ref{bounded}) into Eq.~(\ref{1227}), we find that $\lambda=1$ if and only if 
\begin{equation}
m(o(e))w(e)=m(t(e))w(\bar{e})e^{-i\theta(e)}, 
\end{equation}
for every $e\in D(G)$. Note that this condition is equivalent to the eDBC in Eq.~(\ref{DBC}). 

\segawa{We now show that there exists a $1$-form $\theta$ such that the eDBC holds if and only if $\{|w(e)|^2:e\in D(G)\}$ has the reversible measure. }
Taking its square modulus both sides of Eq.~(\ref{DBC}), we have the sufficiency. 
\HS{We prove the opposite direction; that is, if $\{|w(e)|^2:e\in D(G)\}$ has the reversible measure, 
then there exists a $1$-form $\theta$ such that the eDBC holds. 
Let us consider a spanning tree $\mathbb{T}(G)$ of $G$ and define $C(G)$ as the set of all cycles in $G$. }
\segawa{For $j\in\{1,\dots,r\}$, 
we denote $c_j$ as the cycle generated by adding edge $e_j\in E(G)\smallsetminus E(\mathbb{T}(G))$ to the spanning tree $\mathbb{T}(G)$. 
Here 
\begin{equation*} 
r=|E(G)\smallsetminus E(\mathbb{T}(G))|=|E(G)|-(|V(G)|-1). 
\end{equation*} 
The set of the fundamental cycles is denoted by 
	\begin{equation}\label{fundamental_cycle}C^{(0)}\equiv \{ c_1,c_2,\dots,c_r \}\end{equation}
Thus a one-to-one correspondence 
between $C^{(0)}$ and $E(G)\smallsetminus E(\mathbb{T}(G))$ holds. }
For any $c\in C(G)$, there exist positive integers $\{n_1,n_2,\dots,n_r\}$ such that $c=\sum_{j=1}^{r}n_j c_j$. 
Thus the subset $C^{(0)}$ is a minimum generator of $C(G)$. 
\segawa{We remark that the eDBC holds if and only if for any cycle $c=(e_1,\dots,e_n)\in C^{(0)}$, }
\begin{align}
m(o(e_1)) &= \frac{\widetilde{w}(\bar{e}_1)}{\widetilde{w}(e_1)}m(o(e_2)) \notag  \\
	  &=\frac{\widetilde{w}(\bar{e}_1)\widetilde{w}(\bar{e}_2)}{\widetilde{w}(e_1)\widetilde{w}(e_2)}m(o(e_3))=\cdots 
	  = \prod_{j=1}^n \frac{\widetilde{w}(\bar{e}_j)}{\widetilde{w}(e_j)}m(o(e_1)). \label{mm}
\end{align}
\segawa{Under the assumption that a random walk $\{|w(e)|^2;e\in D(G)\}$ is reversible, Eq.~(\ref{mm}) is equivalent to }
\begin{align} 
\int_{c}\mathrm{arg}(\widetilde{w}) &\in 2\pi \mathbb{Z}, \label{arguments}
\end{align}
where $\int_{c}\mathrm{arg}(\widetilde{f})=\sum_{j=1}^n \left\{\mathrm{arg}(\widetilde{f}(e_j))-\mathrm{arg}(\widetilde{f}(\overline{e_j}))\right\}$. 
\segawa{For a given $w$, the $1$-form $\theta$ can be adjusted to satisfy Eq.~(\ref{arguments}) in the following way:} 
for $c_j=(e_1^{(j)},\dots,e_{k_j}^{(j)})\in C^{(0)}$, 
\begin{equation}
\theta(e_i^{(j)})
	=\begin{cases}
        \int_{c_j}\mathrm{arg}(w) & \text{; $e_i^{(j)}=e_j$,} \\ 
        0 & \text{; otherwise.} 
        \end{cases}
\end{equation}
\end{proof}
Similarly to the proof of Proposition \ref{gururi}, we can show that $\lambda=-1$ if and only if the {\it signed} eDBC holds: 
\[ m(o(e))\widetilde{w}(e)=-m(t(e))\widetilde{w}(\bar{e}). \]
\segawa{We conclude the above discussions by considering-four situations: }
\begin{enumerate}
\renewcommand{\labelenumi}{(\roman{enumi})}
\item\label{case1} $G$ is bipartite, $\{|w(e)|^2:e\in D(G)\}$ is reversible and $\int_{c}\mathrm{arg}(\widetilde{w})\in 2\pi\mathbb{Z}$ for any closed path $c$; 
\item\label{case2} $G$ is non-bipartite, $\{|w(e)|^2:e\in D(G)\}$ is reversible and $\int_{c}\mathrm{arg}(\widetilde{w})\in 2\pi\mathbb{Z}$ for any closed path $c$;
\item\label{case3} $G$ is non-bipartite, $\{|w(e)|^2:e\in D(G)\}$ is reversible and for any closed path $c$, 
\[\int_{c}\mathrm{arg}(\widetilde{w}) \in  
	\begin{cases}
	2\pi\mathbb{Z} & \text{; $c$ is even length closed path, } \\ 
	2\pi(\mathbb{Z}+1/2) & \text{; $c$ is odd length closed path; } 
	\end{cases}
\]
\item\label{case4} otherwise. 
\end{enumerate}
\segawa{The above situations were also considered in Higuchi and Shirai~\cite{HS}, who discussed the spectrum of a twisted random walk on a para-line graph of $G$. }
\segawa{Based on Ref.~\cite{HS}}, we \segawa{state} the following lemma. 
\begin{lemma}\label{multm}
Let $m_{\pm 1}$ be the multiplicities of eigenvalues $\pm 1$ of $T^{(w,\theta)}$, respectively. Then we have 
\begin{equation}
	(m_1,m_{-1})=
        	\begin{cases}
                (1,1) & \text{; case (i),} \\
                (1,0) & \text{; case (ii),} \\
                (0,1) & \text{; case (iii),} \\
                (0,0) & \text{; case (iv).} 
                \end{cases}
\end{equation}
\end{lemma}
\subsection{\segawa{Geometric expression of the subspace $\mathcal{L}^\bot$ in the Grover walk case}} 
\quad \segawa{We investigate a special case of \segawa{the} twisted Szegedy walk with $\theta(e)=0$, $w(e)=1/\sqrt{\mathrm{deg}(o(e))}$, which reduces to the Grover walk. }
We denote the time evolution of the Grover walk by $U^{(w,\theta)}=U_{Grover}$. 
Let $P$ be the probability transition matrix of the symmetric random walk on $G$. 
For any $u,v\in V(G)$, the symmetric means $(P)_{u,v}=\bs{1}_{\{(u,v)\in D(G)\}}/\mathrm{deg}(u)$. 
The reversible distribution of the symmetric random walk $\bs{\pi}: V(G)\to (0,1)$ is given by $\bs{\pi}(u)=\mathrm{deg}(u)/\sum_{v\in V(G)}\mathrm{deg}(v)$. 
Since $T\equiv (T^{(w,\theta)})_{u,v}=\bs{1}_{\{(u,v)\in D(G)\}}/\sqrt{\mathrm{deg}(u)\mathrm{deg}(v)}$ in this setting, we can notice that 
\begin{equation}\label{stadist}
T= \mathfrak{D}^{-1}P\mathfrak{D}, 
\end{equation}
where $\mathfrak{D}=\mathrm{diag}[\sqrt{\bs{\pi}(u)};u\in V(G)]$; 
thus $\mathrm{spec}(P)=\mathrm{spec}(T)$. 
Moreover if $\bs{\eta}$ is an eigenvector of $P$, then $\mathfrak{D}^{-1}\bs{\eta}$ is the eigenvector of $T$ for the same eigenvalue. 

Now let us characterize the eigenspace of the Grover walk corresponding to $\mathcal{L}^{\bot}$ in the above lemma by the cycles of $G$. 
\HS{First, we introduce some new notations.} 
\segawa{We denote the sets of all essential even and odd cycles by $C_e$ and $C_o$, respectively.} 
We also define $C_{o-o}$ as the set of Euler closed path consisting of \segawa{two distinct odd cycles and their so-called bridge}, i.e., 
\begin{multline*} 
C_{o-o}\equiv \{ c=(c_1,p,c_2,p^{-1})\in P(G):\; c_1,c_2\in \segawa{C_o}, p\in \{P_0(G),\emptyset \}, \\  \;o(c_1)=t(c_1)=o(p),\; o(c_2)=t(c_2)=t(p),\; V(c_1)\cap V(c_2)\subseteq \{\emptyset, o(c_1) \}\}. 
\end{multline*} 
\segawa{
\begin{definition}
For $p=(e_1,e_2,\dots,e_n)\in P(G)$, let $\bs{\gamma}, \bs{\tau}: P(G)\to \ell^2(D(G))$ be 
\begin{align}
\bs{\gamma}(p) &=\sum_{j=1}^{n} \left(\delta_{e_j}-\delta_{\bar{e}_j}\right), \label{wawa}\\
\bs{\tau}(p) &=\sum_{j=1}^{n} (-1)^j\left(\delta_{e_j}+\delta_{\bar{e}_j}\right). \label{vava}
\end{align}
\end{definition}
}
\HS{The following theorem relates the subspace $\mathcal{L}^{\bot}$ to a geometric structure of the graph. }
\begin{theorem}\label{cycles}
Define the eigenspaces of $U_{Grover}$ by 
\begin{equation*}
\mathcal{M}_+\equiv \mathcal{L}^{\bot} \cap \mathcal{H}_-^{(S)}\; \mathrm{and}\; 
\mathcal{M}_-\equiv \mathcal{L}^{\bot} \cap \mathcal{H}_+^{(S)}
\end{equation*}
with eigenvalues $1$ and $-1$, respectively. 
Then we have 
	\begin{align}
        \mathcal{M}_+ &= \sum_{c\in C(G)}\mathbb{C}\bs{\gamma}(c), \label{hassei1}\\
	\mathcal{M}_- &= \sum_{c\in C_e \cup C_{o-o}}\mathbb{C}\bs{\tau}(c). \label{hassei2}
	\end{align}
Here $\mathrm{dim}(\mathcal{M}_+)=|E(G)|-|V(G)|+1$, and $\mathrm{dim}(\mathcal{M}_-)=|E(G)|-|V(G)|+\bs{1}_{\{G \mathrm{\;is\;bipartite}\}}$. 
In particular, 
defining $CP_e$ as the set of all even-length closed paths, we have 
\[ \sum_{c\in CP_e}\mathbb{C}\bs{\tau}(c)=\mathcal{M}_-, \;\;\sum_{c\in CP_e}\mathbb{C}\bs{\gamma}(c)\subset \mathcal{M}_+. \]
\end{theorem}
\begin{proof}
From Eqs.~(\ref{alpha}) and (\ref{beta}), we observe that for any $c\in C(G)$, 
we have $\bs{\gamma}(c)\in \mathcal{M}_+$, that is, $U\bs{\gamma}(c)=\bs{\gamma}(c)$. 
\segawa{Let $C^{(0)}$ be the set of the fundamental cycles defined by Eq.~(\ref{fundamental_cycle}).} 
For any $c\in C(G)$, 
there exist positive integers $\{n_1,n_2,\dots,n_r\}$ such that $\bs{\gamma}(c)=\sum_{j=1}^{r}n_j \bs{\gamma}(c_j)$. 
Therefore, we arrive at $\mathcal{M}_+=\mathrm{span}\{ \bs{\gamma}(c); c\in C^{(0)} \}=\mathrm{span}\{ \bs{\gamma}(c); c\in C(G) \}$. 

Next, from Eqs.~(\segawa{\ref{alpha2}}) and (\ref{beta}), for any $c\in C_e(G)\cup C_{o-o}$, 
we readily confirm $\bs{\tau}(c)\in \mathcal{M}_-$. So we have $U\bs{\tau}(c)=-\bs{\tau}(c)$. 
Note that $C^{(0)}\subseteq C_e(G)$ if and only if the graph is \segawa{bipartite}. 
In this case, $\mathcal{M}_-=\mathrm{span}\{\bs{\tau}(c): c\in C^{(0)}\}$. 
Now we consider that $C^{(0)}$ also contains odd cycles. 
\sato{Let odd cycles be denoted by $c_1,\dots,c_K$ and the even cycles by $c_{K+1},\dots, c_{r}$ $(K\leq r)$.} 
We remark that for any \segawa{cycles} $c_i,c_j\in C_o(G)$, there exists a simple path $p_{i,j}\in P_0(G)$ such that $(c_i,p_{i,j},c_j,p_{i,j}^{-1})\in C_{o-o}$. 
Then 
\begin{equation} 
\mathcal{M}_- =\mathrm{span}\{\{\bs{\tau}(c_1,p_{1,j},c_{j},p_{1,j}^{-1}): \;2\leq j\leq K \} \cup \{ \bs{\tau}(c_j): K<j\leq r  \}\}. 
\end{equation}
Since all the $\bs{\tau}(c_1,p_{1,j},c_{j},p_{1,j}^{-1})$ $(2\leq j\leq K)$ and $\bs{\tau}(c_j)$ $(K+1 \leq j \leq r)$ are linearly independent, 
we have \[ \mathrm{dim}(\mathcal{M}_-)=K-1+(r-K)=|E|-|V|. \] 
We can easily check that $\bs{\tau}(p)\in \mathcal{M}_-$ and $\bs{\gamma}(p)\in \mathcal{M}_+$ for any $p\in CP_e$. 
Since $C_{o-o},C_e\subset CP_e$, we obtain the conclusion.  
\end{proof}


\section{Grover walk on crystal lattice}
\subsection{Setting}
We define a partition $\pi: G\to G^{(o)}=(V^{(o)},D^{(o)})$ satisfying the following conditions: 
\begin{enumerate}
\item $V^{(o)}=\pi(V(G))=\{ V_1,V_2,\dots, V_r \}$ with $V_i\cap V_j=\emptyset$ $(i\neq j)$. 
\item for $e\in D(G)$, $\pi(o(e))=o(\pi(e))$ and $\pi(t(e))=t(\pi(e))$.
\item $D_u \stackrel{\pi}{\to } D_{\pi(u)}$: \segawa{a bijection}, \\ \;\;\;\; where $D_u=\{e\in D(G): o(e)=u \}$. 
\end{enumerate}
\begin{remark}
Under the assumption of (3), $\mathrm{deg}(u)=\mathrm{deg}(\pi(u))$ for all $u\in V(G)$. 
\end{remark}
In particular, if there exists an abelian \segawa{group} $\Gamma\subset \mathrm{Aut}(G)$ such that 
\begin{align*} 
	V^{(o)} &= \{\Gamma\mathrm{-orbit\;of\;}G\}, \\
	(\pi(u),\pi(v))\in D^{(o)} &\Leftrightarrow \mathrm{\;there\;exists\;} g\in\Gamma \mathrm{\;such\;that\;} (u,gv)\in D(G),
\end{align*}
then 
$G$ is called \segawa{a} crystal lattice. 
We denote a spanning tree by $\mathbb{T}^{(o)}$ of $G^{(o)}=\Gamma\backslash G$, and let $E(G^{(o)})\smallsetminus E(\mathbb{T}^{(o)})=\{e_1,\dots,e_r\}$. 
We assign \segawa{a} $1$-form $\theta$ such that $\theta(e)=0$ if and only if $e\in D(\mathbb{T}^{(o)})$. 
Prepare $r$ vectors $\hat{\theta}_1,\dots,\hat{\theta}_r\in \mathbb{R}^d$ $(d\leq r)$, 
where $r=|E(G^{(o)})\smallsetminus E(\mathbb{T}^{(o)})|$. 
\segawa{Later in the discussion, we will specify $\theta(e_j)=-\theta(\bar{e}_j)=\langle k,\hat{\theta}_j \rangle$ for $k\in\mathbb{R}^d$ and 
$e_j\in E(G^{(o)}\smallsetminus \mathbb{T}^{(o)})$.  }

\segawa{Given a finite graph $G^{(o)}$, we construct an abelian covering graph $G$. }
\segawa{We prepare $r$-unit vector $\hat{\theta}_1,\dots,\hat{\theta}_r\in \mathbb{R}^d$ $(d\leq r)$. 
with $\mathrm{rank}[\hat{\theta}_1,\dots,\hat{\theta}_r]=d$. } 
First, we copy the quotient graph $G^{(o)}$ to each lattice 
$L=\mathbb{Z}\hat{\theta}_1+\cdots+\mathbb{Z}\hat{\theta}_r\subset \mathbb{R}^d$. 
Here the vertex $u$ of $G^{(o)}$ at $\mathbf{x}\in L$ is labeled $(\mathbf{x},u)$. 
Let $\mathbb{T}^{(o)}$ be a spanning tree of $G^{(o)}$. 
Secondly, we rewire the terminus of each arc $e_j\in D(G^{(o)}\smallsetminus \mathbb{T}^{(o)})$ at $\mathbf{x}\in L$ to the terminus of its neighbor
located in $\mathbf{x}+\hat{\theta}_j$; \segawa{for instance} 
if $t(e_j)=(\mathbf{x},v)$ in the first step, then, $t(e_j)$ \segawa{becomes} $(\mathbf{x}+\hat{\theta}_j,v)$ in the second step. 
\segawa{This procedure is repeated for each lattice $\mathbf{x}\in L$ to obtain a covering graph $G$.} 
In particular, 
when the transformation group $\Gamma$ with $G^{(o)}=\Gamma\backslash G$ is the 1-homology group $H_1(G^{(o)},\mathbb{Z})$, 
$G$ is called \segawa{the} maximal abelian covering graph of $G^{(o)}$. 
In other words, when $\{\hat{\theta}_j\}_{j=1}^r$ are linearly independent, then $G$ is the maximal abelian covering graph. 
\segawa{The fundamental domain $F^{(o)}$ of the abelian covering graph $G$ is denoted by 
	\begin{multline} F^{(o)}=\{\mathrm{the \; spanning \; tree \;} \mathbb{T}^{(o)} \mathrm{\;embedded \; on \;}\bs{o} \in L\} \\
        	\cup \{\mathrm{the \; rewired \; arcs \;from \;} D(G^{(o)}\smallsetminus \mathbb{T}^{(o)}) \mathrm{\;on \;} \bs{o}\in L\}. 
        \end{multline}}
\segawa{We take $V^{(o)}=V(F^{(o)})$ and $D^{(o)}=V(F^{(o)})$. }
\HS{In other words, $\sigma_1\neq \sigma_2\in\mathrm{Aut}(G)$, 
\[
\sigma_1(V(F))\cap \sigma_2(V(F)) = \emptyset, \;\;\sigma_1(D(F))\cap \sigma_2(D(F)) = \emptyset
\]
and $V(G)=\bigcup_{\sigma\in\mathrm{Aut}(G)}\sigma(V(F))$,\;$D(G)=\bigcup_{\sigma\in\mathrm{Aut}(G)}\sigma(D(F))$. 
}
\segawa{By the above observation, given a crystal lattice $G=(V,D)$, we have 
 \[ V(G)\cong L\times V^{(o)} \;\mathrm{and}\; D(G)\cong L\times D^{(o)}, \] 
respectively. }

We remark that the coin operator $C$ is reexpressed as $\bigoplus_{u\in V(G)}H_u$, 
where $H_u$ is so called the Grover coin operator on $\mathcal{H}_u$ assigned at the vertex $u$: 
\[ (H_u)_{e,f}=\bs{1}_{\{o(e)=o(f)=u\}}\left( \frac{2}{\mathrm{deg}(u)}-\delta_{e,f} \right). \]
We assume that 
$(H_u)_{e,f}=(H_{\sigma(u)})_{\sigma(e),\sigma(f)} \equiv (H^{(o)}_{\pi(u)})_{e_o,f_o}$ 
for any $e\in \pi^{-1}(e_o)$, $f\in \pi^{-1}(f_o)$, and $\sigma\in \Gamma$. 
Let $\Psi_n:L\times V^{(o)} \to \ell^2(L\times D^{(o)})$ at time $n\in \mathbb{N}$ be denoted by $\Psi_n(\mathbf{x},v)=\Pi_{\mathcal{H}_{v}}U^n\Psi_0$. 
Here $\Psi_0\in \ell^2(L\times D^{(o)})$ is \HS{an} initial state. 
From Eq.~(\ref{analogue}), we have 
\begin{equation}\label{receq}
\Psi_{n}(\mathbf{x},v)=\sum_{f\in D^{(o)}:t(f)=v}P_f\Psi_{n-1}\left(\mathbf{x}-\hat{\theta}(f),o(f)\right),  
\end{equation}
where 
\segawa{
 \begin{equation*}
 \hat{\theta}(f)=
 \begin{cases} 
 0 & \text{: $f\in E(\mathbb{T}^{(o)})$,} \\  \hat{\theta}_j & \text{: $f=e_j\in E(G^{(o)}\smallsetminus\mathbb{T}^{(o)}),\;\;(j\in \{1,\dots,r\})$. }
 \end{cases} 
 \end{equation*}
}
\HS{We define the finding distribution $\mu_n:L\to [0,1]$ by 
\[ \mu_n(\mathbf{x})=\sum_{u\in V^{(o)}}\nu_n(\mathbf{x},u), \]
and we denote by $X_n$ a random variable following $P(X_n=\mathbf{x})=\mu_n(\mathbf{x}).$ 
Our interest is devoted to the sequence of $\{\mu_n\}_{n\in \mathbb{N}}$} in the rest of this paper. 
%
\subsection{Spectrum}
Let the Fourier transform $\mathcal{F}: L^2(K^d\times D^{(o)})\to \ell^2(L\times D^{(o)})$ be 
\[ (\mathcal{F}g)(\mathbf{x},f)=\int_{0\;d}^{2\pi} g(\mathbf{k},f) e^{-i\langle \mathbf{k},\mathbf{x} \rangle} \frac{d\mathbf{k}}{(2\pi)^d}. \]
Here $K=[0,2\pi)$. 
The dual operator $\mathcal{F}^*: \ell^2(L\times D^{(o)})\to L^2(K^d\times D^{(o)})$ is described by   
\[ (\mathcal{F}^*\phi)(\mathbf{k},f)=\sum_{\mathbf{x}\in L}\phi(\mathbf{x},f)e^{i\langle \mathbf{k},\mathbf{x} \rangle}. \]
Taking $\mathcal{F}^*$ to both sides of Eq.~(\ref{receq}), we have 
\begin{align}\label{rec}	
	\widetilde{\Psi}_n(\mathbf{k},v) = \sum_{ f: t(f)=v } e^{ i\langle \mathbf{k},\hat{\theta}(f) \rangle }P_f \widetilde{\Psi}_{n-1}(\mathbf{k},o(f)). 
\end{align}
Here $\widetilde{\Psi}_{n}(\mathbf{k},v)=\mathcal{F}^*\Psi_{n}(\mathbf{x},v)$. 
Define $\widetilde{U}^{(o)}_{\mathbf{k}}$ as the twisted Szegedy walk $U^{(w,\theta)}$ on the quotient graph $G^{(o)}$ with 
\begin{equation}\label{crystal_case}
	w(e)=1/\sqrt{\mathrm{deg}(o(e))},\;\; \theta(e)=\langle \mathbf{k},\hat{\theta}(e) \rangle \;\;(e\in D^{(o)}). 
\end{equation}
We also denote the discriminant operator of $\widetilde{U}^{(o)}_{\mathbf{k}}$ by $\widetilde{T}^{(o)}_{\mathbf{k}}\equiv T^{(w,\theta)}$ 
with the above $w$ and $\theta$. 

The important statement of this section is that the Grover walk on $G$ is reduced to the twisted Szegedy walk 
on the quotient graph $G^{(o)}$ in the Fourier space as follows: 
\begin{equation}\label{FT}
	\widetilde{\Psi}_n(\mathbf{k},v)= \Pi_{\mathcal{H}_{v}} \left(\widetilde{U}^{(o)}_{\mathbf{k}}\right)^n \sum_{u\in V^{(o)}} \widetilde{\Psi}_0(\mathbf{k},u). 
\end{equation}
In other words, if $U_{Grover}$ is the time evolution of the Grover walk on \segawa{the abelian covering graph of $G^{(o)}$}, 
then we have for any $\widetilde{\Psi}_0\in L^2(K^d\times D^{(o)})$, 
\begin{equation}\label{FT2}
	\mathcal{F}[\{\widetilde{U}^{(o)}_{\mathbf{k}}\}^n\widetilde{\Psi}_0]= U_{Grover}^n \mathcal{F}[\widetilde{\Psi}_0]. 
\end{equation}

On the other hand, let $Y_n^{(u)}\in L\times V^{(o)}$ be a simple random walk at time $n$ on $G$ starting from $(0,u)$. 
Denote the characteristic function of $Y_n^{(u)}$, for $\mathbf{k}\in K^d$, by
\[ \chi_n(\mathbf{k};u,v)= \sum_{\mathbf{x}\in L} P\left(Y_n^{(u)}=(\mathbf{x},v)\right)e^{i\langle \mathbf{k}, \mathbf{x} \rangle}. \]
Let $|V^{(o)}|\times |V^{(o)}|$ matrix $\left(\tilde{P}_{\mathbf{k}}^{(o)}\right)_{u,v}$ for $u,v\in V^{(o)}$ be 
$\left(\tilde{P}_{\mathbf{k}}^{(o)}\right)_{u,v}=\chi_1(\mathbf{k};u,v)$. 
Then \[\chi_n(\mathbf{k};u,v)=\left(\left\{\tilde{P}_{\mathbf{k}}^{(o)}\right\}^n\right)_{u,v} . \]
$\tilde{T}^{(o)}_{\mathbf{k}}$ and $\tilde{P}^{(o)}_{\mathbf{k}}$ are related by 
\begin{equation}\label{tM}
\tilde{T}^{(o)}_{\mathbf{k}}= \mathfrak{D}^{-1}\tilde{P}^{(o)}_{\mathbf{k}}\mathfrak{D}, 
\end{equation}
where $\mathfrak{D}$ is given by Eq.~(\ref{stadist}) inserting the square root of the stationary distribution 
of \segawa{the} simple random walk on $G^{(o)}$; that is, \sato{$\sqrt{\bs{\pi}(u)}=\sqrt{\mathrm{deg}(u)}$} for all $u\in V^{(o)}$. 
\begin{proposition}\label{cristalspec}
Let $\tilde{P}_{\mathbf{k}}^{(o)}$ be defined by the above. 
When $\mathbf{k}\in \mathbb{R}^d$ satisfies $\langle \mathbf{k},\hat{\theta}(e)\rangle\notin \pi \mathbb{Z}$, for some $e\in D^{(o)}$, 
then we have 
\[ \mathrm{spec}(\widetilde{U}_{\mathbf{k}}^{(o)})= \varphi_{QW}^{-1}(\mathrm{spec}(\tilde{P}_{\mathbf{k}}^{(o)})) \cup \{ 1 \}^{|E^{(o)}|-|V^{(o)}|} \cup \{ -1 \}^{|E^{(o)}|-|V^{(o)}|}. \] 
\end{proposition}
\begin{proof}
From Eq.~(\ref{tM}), we have $\mathrm{spec}(\tilde{T}^{(o)}_{\mathbf{k}})=\mathrm{spec}(\tilde{P}^{(o)}_{\mathbf{k}})$, 
\segawa{and the eigenfunctions $\rho$ and $\rho'$ with the same eigenvalue of $\tilde{T}^{(o)}_{\mathbf{k}}$ and $\tilde{P}^{(o)}_{\mathbf{k}}$, respectively, satisfy $\rho'=\mathfrak{D}\rho$. }
Therefore we can directly apply Lemma~\ref{key}.  
On the other hand, the assumption $\langle \mathbf{k},\hat{\theta}(e)\rangle\notin \pi \mathbb{Z}$, for some $e\in D^{(o)}$ implies case (iv). 
From Lemma~\ref{multm}, $m_{\pm 1}=0$; that is, the multiplicities of $\pm 1$ are $M_{\pm 1}=|E^{(o)}|-|V^{(o)}|$, which concludes the proof. 
\end{proof}
Let the eigenspace of eigenvalues $\varphi_{QW}^{-1}(\mathrm{spec}(\widetilde{T}^{(o)}_{\mathbf{k}}))$ 
be $\widetilde{\mathcal{L}}^{(o)}_{\mathbf{k}}\subset L^2(K^d\times D^{(o)})$. 
We denote $\widetilde{S}^{(o)}_{\mathbf{k}}:  L^2(K^d\times D^{(o)})\to  L^2(K^d\times D^{(o)})$ by 
$\widetilde{S}^{(o)}_{\mathbf{k}}f(\mathbf{k},e)=e^{-i\langle \mathbf{k},\hat{\theta}(e) \rangle}f(\mathbf{k},\bar{e})$. 
Remark that $\widetilde{S}^{(o)}_{\mathbf{k}}$ has two eigenvalues $\pm 1$. 
We denote the eigenspaces of eigenvalues $\pm 1$ for $\widetilde{S}^{(o)}_{\mathbf{k}}$ by $\widetilde{\mathcal{H}}_{\pm}^{(o)}$ and 
define $\tilde{\mathcal{M}}_\pm^{(o)} = \widetilde{\mathcal{L}_\mathbf{k}}^{(o)} \cap \widetilde{\mathcal{H}}_{\mp}^{(o)}$. 
\sato{
\begin{definition}
For a given crystal lattice $G$, we define $\mathcal{C}_\pm\subset \ell^2(L\times D^{(o)})$ by 
	\begin{align*} 
        \mathcal{C}_+ &= \sum_{c\in C(G)} \mathbb{C}\gamma(c), \\
        \mathcal{C}_- &= \sum_{c\in CP_e(G)} \mathbb{C}\tau(c).
        \end{align*}
\end{definition}
We remark that $\mathcal{M}_\pm$ expressed by Eqs.~(\ref{hassei1}) and (\ref{hassei2}) are defined for a finite graph. }
\begin{proposition}\label{fouriecycle}
\HS{For} a crystal graph $G$, 
it holds that
	\[\sum_{\mathbf{x}\in L}\mathbb{C}\mathcal{F}(e^{i\langle \mathbf{k}, \mathbf{x}\rangle}\tilde{\mathcal{M}}_{\pm}^{(o)})=\mathcal{C}_{\pm}.\] 
\end{proposition}
\begin{proof}
First, we prove that 
$\mathcal{C}_{\pm}\subset \sum_{\mathbf{x}\in L}\mathbb{C}\mathcal{F}(e^{i\langle \mathbf{k}, \mathbf{x}\rangle}\tilde{\mathcal{M}}_{\pm}^{(o)})$. 
We remark that \segawa{a} closed cycle in $G$, $c=(f_1,f_2,\dots, f_\kappa)$, where $f_j=(x_j,e_j)$ ($x_j\in L$, $e_j\in D^{(o)}$), 
is represented by \segawa{the} closed path $(e_1,\dots,e_\kappa)$ in $G^{(o)}$. 
Note that 
\[ \bar{f}_j=(x_j+\hat{\theta}(e_j),\bar{e}_j)\;\mathrm{and}\;{f}_{j+1}=(x_j+\hat{\theta}(e_j),{e}_{j+1}). \]
It holds that 
\begin{align}
	(\mathcal{F}^*\bs{\gamma}(c))(\mathbf{k},\bar{e}_j) &= -e^{i\theta(e_j)}(\mathcal{F}^*\bs{\gamma}(c))(\mathbf{k},{e}_j), \;\;(1\leq j\leq \kappa)\label{w1}\\
	(\mathcal{F}^*\bs{\gamma}(c))(\mathbf{k},e_{j+1}) &= e^{i\theta(e_j)}(\mathcal{F}^*\bs{\gamma}(c))(\mathbf{k},{e}_j). \;\;(1\leq j<\kappa) \label{w2}
\end{align}
Here we define $\theta(e_j)=\langle \mathbf{k},\hat{\theta}(e_j) \rangle$. (See Eq.~(\ref{crystal_case}).)  
From Eqs.~(\ref{w1}) and (\ref{w2}), we have 
\[ (\mathcal{F}^*\bs{\gamma}(c))(\mathbf{k},\bar{e}_{\kappa}) = -\exp\left[ i\int_{(e_1,\dots,e_\kappa)} \theta\right] (\mathcal{F}^*\bs{\gamma}(c))(\mathbf{k},{e}_1). \]
Since $c$ is \segawa{a} closed cycle of $G$, we have $\int_{(e_1,\dots,e_\kappa)} \theta=0$, and therefore 
\begin{equation}\label{w3}
	\sum_{f:o(f)=o(e_j)}(\mathcal{F}^*\bs{\gamma}(c))(f)=0, \;\;(1\leq j\leq \kappa).  
\end{equation}
Equation (\ref{w1}) implies that 
\begin{equation}\label{w4}
\mathcal{F}^*\bs{\gamma}(c)\in \sato{\widetilde{\mathcal{H}}_-^{(o)}}. 
\end{equation}
Combining Eq.~(\ref{w3}) with Eq.~(\ref{w4}), 
we have $\mathcal{F}^*\bs{\gamma}(c)\in \tilde{\mathcal{M}}^{(o)}_+$ from Eq.~(\ref{beta}). 
Similarly, for any even-length closed path $c$ in $G$, 
we have $\mathcal{F}^*\bs{\tau}(c)\in \tilde{\mathcal{M}}^{(o)}_-$. 
Thus we find that 
$\mathcal{C}_{\pm}\subset \sum_{\mathbf{x}\in L}\mathbb{C}\mathcal{F}(e^{i\langle \mathbf{k}, \mathbf{x}\rangle}\tilde{\mathcal{M}}_{\pm}^{(o)})$. 

We now prove that 
$\mathcal{C}_{\pm}\supset \sum_{\mathbf{x}\in L}\mathbb{C}\mathcal{F}(e^{i\langle \mathbf{k}, \mathbf{x}\rangle}\tilde{\mathcal{M}}_{\pm}^{(o)})$. 
The set of cycles $C^{(o)}=\{c_1,\dots,c_r\}$ with $r=|E^{(o)}|-|V^{(o)}|+1$ \HS{can be identified with} $E(G^{(o)}\smallsetminus \mathbb{T}^{o})$. 
\segawa{Here we assume that each $c_j=(e_1^{(j)},\dots,e_{n_j}^{(j)})$ $(j\in\{1,\dots,r\})$ satisfies 
	\[ \theta(e_i^{(j)})= \begin{cases} \theta_j & \text{: $i=1$,} \\ 0 & \text{: otherwise.} \end{cases}\]
}
\segawa{For the pair of $c_1$ and $c_j$, a path $P=(f_1,\dots,f_l)$ satisfying $o(P)=o(c_1)$ and $t(P)=o(c_j)$ exists. } 
Define new closed paths $\tilde{c_1}$ and $\tilde{c_j}$ by
\begin{align*}
\tilde{c_1}&=(\bar{f}_{\kappa},\bar{f}_{\kappa-1},\dots,\bar{f}_{1},e_1^{(1)},\dots,e_{n_1}^{(1)},f_{1},\dots,f_{\kappa-1},f_{\kappa}), \\
\tilde{c_j}&=(f_{\kappa+1},f_{\kappa+2},\dots,f_{l},e_1^{(j)},\dots,e_{n_j}^{(j)},\bar{f}_{l},\bar{f}_{l-1},\dots,\bar{f}_{\kappa+1}),\;\;(1\leq\kappa\leq l)
\end{align*}
respectively.
Denote $\eta_j, \zeta_j\in \ell^2(D^{(o)})$ by 
\begin{align*}
	\eta_1(e) 
        & = \begin{cases} 
                        (-1)^{\kappa-m} & \text{; $e=\bar{f}_m$, $1\leq m\leq \kappa$} \\
                        (-1)^{\kappa} & \text{; $e=e^{(1)}_1$} \\
                        (-1)^{\kappa+m-1}e^{i\theta_1} & \text{; $e=e^{(1)}_m$, $2\leq m\leq n_1$} \\
        	        (-1)^{\kappa+n_1-1+m} e^{i\theta_1} & \text{; $e=f_m$, $1\leq m\leq \kappa$} \\
                         0 & \text{; otherwise}
           \end{cases}, \\
        \zeta_1(e) 
        &= \begin{cases}  1 & \text{; $e=\bar{f}_m$, $1\leq m\leq \kappa$ or $e=e^{(1)}_1$} \\
                         e^{i\theta_1} & \text{; $e={e^{(j)}_m}$, $1\leq m\leq n_1$ or $e=f_s$, $1\leq s\leq \kappa$} \\
                         0 & \text{; otherwise}
           \end{cases}
\end{align*}
For $2\leq j\leq r$, 
\begin{align*}
	\eta_j(e) 
        &= \begin{cases} 
                        (-1)^{m-\kappa-1} & \text{; $e={f}_m$, $\kappa+1\leq m\leq l$} \\
                        (-1)^{l-\kappa} & \text{; $e={e^{(j)}_1}$} \\
                        (-1)^{l-\kappa+m-1}e^{i\theta_j} & \text{; $e=e^{(j)}_m$, $2\leq m\leq n_j$} \\
        	        (-1)^{l-\kappa+n_j+(l-m)} e^{i\theta_j} & \text{; $e=\bar{f}_m$, $\kappa+1\leq m\leq l$} \\
                         0 & \text{; otherwise}
           \end{cases}, \\
        \zeta_j(e) 
        &= \begin{cases}  1 & \text{; $e=f_m$, $\kappa+1 \leq m\leq l$ or $e=e^{(j)}_1$} \\
                         e^{i\theta_j} & \text{; $e={e^{(j)}_m}$, $1\leq m\leq n_j$ or $e=\bar{f}_s$, $\kappa+1\leq s\leq l$} \\
                         0 & \text{; otherwise}
           \end{cases}
\end{align*}
We can notice that taking $Q_{\pm}\equiv I\pm S^{(\theta)}$, for $u\in V(\widetilde{c}_j)$, 
\begin{align} 
\sum_{e: o(e)=u}Q_+\eta_j(e)
	&= \begin{cases}
        1-(-1)^{n_j} e^{i\theta_j} & \text{; $u=o(\widetilde{c}_j)$} \\
        0 & \text{; otherwise} 
        \end{cases} \\
\sum_{e: o(e)=u}Q_-\zeta_j(e)
	&= \begin{cases}
        1-e^{i\theta_j} & \text{; $u=o(\widetilde{c}_j)$} \\
        0 & \text{; otherwise} 
        \end{cases}         
\end{align}
\HS{Moreover it holds that $Q_+\eta \in \tilde{\mathcal{H}}_+^{(o)}$ and $Q_-\eta \in \tilde{\mathcal{H}}_-^{(o)}$ since $S^{(\theta)}Q_\pm=\pm Q_\pm$.} 
From these observations, \segawa{we replace $\eta_j$ and $\zeta_j$ 
with $\tilde{\eta}_j$ and $\tilde{\zeta}_j$ so that $\tilde{\eta}_j,\tilde{\zeta}_j\in \mathrm{ker}(d^*_A)$: }
\begin{align} 
\tilde{\eta}_j &= \left(1-(-1)^{n_1}e^{i\theta_1}\right)Q_+\eta_j-\left(1-(-1)^{n_j}e^{i\theta_j}\right)Q_+\eta_1, \\
\tilde{\zeta}_j &= \left(1-e^{i\theta_1}\right)Q_-\zeta_j-\left(1-e^{i\theta_j}\right)Q_-\zeta_1. 
\end{align}
Thus Eq.~(\ref{beta}) implies $\tilde{\eta}_j\in \tilde{\mathcal{M}}_-^{(o)}$, $\tilde{\zeta}_j\in \tilde{\mathcal{M}}_+^{(o)}$, respectively. 
Since the functions $\{\tilde{\eta}_j\}_j^{r}$ and $\{\tilde{\zeta}_j\}_j^{r}$ are linearly independent and 
the situation is the case (iv) in Eq.~(\ref{multm}) for almost all $\mathbf{k}\in \mathbb{R}^d$, it holds that 
$\tilde{\mathcal{M}}_-^{(o)}=\mathrm{span}\{\tilde{\eta}_j;j\in \{1,\dots,r\}\}$ and 
$\tilde{\mathcal{M}}_+^{(o)}=\mathrm{span}\{\tilde{\zeta}_j;j\in \{1,\dots,r\}\}$ a.e. 
We can confirm that for the pair of closed paths $\tilde{c}_1$ and $\tilde{c}_j$ in $G^{(o)}$, 
there exists a \segawa{finite and even-length closed path} $p$ in $G$ such that 
$\mathcal{F}\tilde{\zeta}_j=\gamma(p)$ and $\mathcal{F}\tilde{\eta}_j=\tau(p)$. 
From Theorem~\ref{cycles}, 
we have $\sum_{\mathbf{x}\in L}\mathbb{C}\mathcal{F}(e^{i\langle \mathbf{k}, \mathbf{x}\rangle}\tilde{\mathcal{M}}_{\pm}^{(o)})\subset \mathcal{C}_\pm$. 
\end{proof}
Let \sato{$P_G: V\times V\to [0,1]$} be the stochastic operator of a simple random walk on \segawa{a} crystal lattice $G$; that is, 
\[ P_G(u,v)=\boldsymbol{1}_{(u,v)\in D(G)}/\mathrm{deg}(u). \]
Let $U(G)$ be the time evolution of the Grover walk on crystal lattice $G$. As a consequence of Propositions \ref{cristalspec} and \ref{fouriecycle}, 
we have the following corollary.
\begin{corollary}
\HS{For a crystal lattice $G$, we put} 
\[N=\begin{cases} \infty & \text{; $G$ has cycles,} \\ 0 & \text{; otherwise.} \end{cases}\]
Then we have 
\begin{equation}
\mathrm{spec}(U(G))=\varphi^{-1}_{QW}(\HS{P_G})\cup \{1\}^{N} \cup \{-1\}^{N}
\end{equation}
\end{corollary}
\subsection{Stochastic behaviors}
Here we define two specific properties of quantum walks; localization and linear spreading: 
\segawa{
\begin{definition}
\noindent 
\begin{enumerate}
\item We say that localization occurs if there exists $\mathbf{x}\in L$ such that 
\[ \limsup_{n\to\infty}\mu_n (\mathbf{x})>0. \]
\item We say that linear spreading occurs if 
\[ \lim_{n\to\infty} \sum_{\mathbf{x}\in L}||\mathbf{x}/n||^2 \mu_n (\mathbf{x}) \in (0,\infty) \]
\end{enumerate}
\end{definition}
}
We also define 
\[ \Phi^{(o)}=\{\phi: \mathbb{R}^d\to \mathbb{R}^d: \cos\phi(\mathbf{k})\in \sato{\mathrm{spec}(\tilde{P}^{(o)}_{\mathbf{k}})}\}. \]
Let \segawa{the} eigenfunction of \segawa{an} eigenvalue $\cos\phi$ with $\phi\in \Phi^{(o)}$ 
be $\omega_\phi\in L^2(K^d\times V^{(o)})$ with \segawa{$\sum_{u\in V^{(o)}}|\omega_\phi(\mathbf{k},u)|^2=1$}. 
We put the set of \segawa{the} arccos's of the constant eigenvalues 
\begin{align*} 
\Lambda &= \{\phi\in \Phi^{(o)}: \partial\phi(\mathbf{k})/\partial k_1=\partial\phi(\mathbf{k})/\partial k_2=\cdots=\partial\phi(\mathbf{k})/\partial k_d=0,\mathrm{\;for\;all}\;\mathbf{k}\in K^d \} \\
 &=\{\pm \alpha_1,\dots,\pm \alpha_s\},\;(0\leq s\leq |V^{(o)}|). 
\end{align*}
Here $s=0$ indicates $\Lambda=\emptyset$. 
Let $\mathcal{S}_\phi$ be the set of all critical points of $\phi\in \Phi^{(o)}\smallsetminus \Lambda$ on $K^d$; 
that is, $\partial \phi|_{\mathcal{S}_\phi}=0$. 
\segawa{To demonstrate the above properties, we adopt the stationary phase method. }
We impose the following natural assumption on $\mathcal{S}_\phi$ :

\begin{assumption}\label{assumption}
For every $\phi\in \Phi^{(o)}\smallsetminus \Lambda$, 
\begin{enumerate}
\renewcommand{\labelenumi}{(\alph{enumi})}
\item both $\phi(\mathbf{k})$ and the eigenfunction $\bs{w}_\phi(\mathbf{k},e)$ of eigenvalue $e^{i\phi(\mathbf{k})}$ are analytic on a neighbor of $\mathcal{S}_\phi$; 
\item each critical point $p\in \mathcal{S}_\phi$ is 
non-degenerate; that is, its Hessian matrix is invertible at $p$. 
Here the Hessian matrix of $\phi$ at $p$ is defined by 
\[ \left(\mathrm{Hess}_\phi(p)\right)_{l,m}=\frac{\partial}{\partial k_l \partial k_m}\phi(\mathbf{k})\bigg|_{\mathbf{k}=p}. \]
\end{enumerate}
\end{assumption}
\segawa{The stationary phase satisfies the following lemma (see for example \cite{PW,Ste}). }
\begin{lemma}\label{stationary}
Let $\psi$ and $\phi$ with $\mathrm{Re}(\varphi)\geq 0$ 
be complex-valued analytic functions on a compact neighborhood $\mathcal{N}$ of the origin in $\mathbb{R}^d$. 
Suppose that $\phi$ has a single critical point that is non-degenerate at $\mathbf{k}_0$. 
Then for sufficiently large $\lambda$ we have, 
\[ \int_{\mathcal{N}} e^{\lambda\phi(\mathbf{k})}\psi(\mathbf{k}) d\mathbf{k} 
	\sim \left(\frac{2\pi}{\lambda}\right)^{d/2} \frac{\psi(\mathbf{k}_0)e^{i\lambda\phi(\mathbf{k}_0)}}{\sqrt{\mathrm{Hess}_\phi(\mathbf{k}_0)}} \]
\segawa{Here the signature of the square root is decided by subtracting the number of the negative eigenvalues of $\mathrm{Hess}_\varphi(\mathbf{k})$ from the number of the positive eigenvalues. }
If $\phi$ has no critical points on the domain, then 
$\int_{\mathcal{N}} e^{\lambda\phi(\mathbf{k})}\psi(\mathbf{k}) d\mathbf{k}=O(\lambda^{-N})$ for any $N>0$. 
\end{lemma} 
Typical examples of the crystal lattice are $\mathbb{Z}^d$, the triangular lattice and the hexagonal lattice. 
In the next section, we show that all three of these lattices satisfy assumptions (a) and (b). 
\segawa{Investigating the general properties of $\Phi^{(o)}$ is outside the scope of this paper, but remains an interesting future's problem.}
In the discrete-time quantum walk on \segawa{the} triangular lattice proposed in \cite{KSKJ}, 
\segawa{which differs from that presented here,  
by the {\it degenerate} critical points, 
the decay rate of the return probability $p_n$ in terms of time $n$ becomes slow down: 
$p_n\propto n^{-4/3}$ for large time step $n$, suggesting that the spreading rate is less than unity. }

By using Lemma~\ref{stationary}, we have the following theorem related to \segawa{localization}.
\begin{theorem}\label{loc_cycle}
Under Assumption~\ref{assumption} (a) and (b), if the quotient graph $G^{(o)}$ satisfies $|E^{(o)}|-|V^{(o)}|\geq 1$ or 
the discriminant operator $\widetilde{T}^{(o)}_{\mathbf{k}}$ has a constant eigenvalue with respect to $\mathbf{k}\in \mathbb{R}^d$, 
then an appropriate choice of the initial state ensures localization of the Grover walk on $G$. 
\end{theorem}
\begin{proof}
Suppose that the initial state of the Grover walk on \segawa{the} crystal lattice $G$ is 
$\Psi_0\in \ell^2(L\times D^{(o)})$ and that $\mathcal{F}^*\Psi_0=\widetilde{\Psi}_0(\mathbf{k})$. 
\segawa{Applying} $\mathcal{F}$ on \sato{$\widetilde{\Psi}_n(\mathbf{k},u)$}, we get 
\[ \Psi_n(\mathbf{x},u) = \int_{0\;d}^{2\pi} e^{-i\langle \mathbf{k},\mathbf{x} \rangle}\widetilde{\Psi}_n(\mathbf{k},u) \frac{d\mathbf{k}}{(2\pi)^d}. \]
Define $\widetilde{\Psi}_n(\mathbf{k})=\sum_{u\in V^{(o)}}\widetilde{\Psi}_n(\mathbf{k},u)$. 
From Lemma~\ref{key} and Eq.~(\ref{FT}), 
we have 
\begin{multline}\label{yume}
\widetilde{\Psi}_n(\mathbf{k})
	=\sum_{\phi\in \Phi^{(o)}\smallsetminus \Lambda}e^{in \phi(\mathbf{k})} \Pi_{\omega_\phi}\widetilde{\Psi}_0(\mathbf{k})\\
               +\sum_{\alpha \in \Lambda}e^{in \alpha} \Pi_{\omega_\alpha}\widetilde{\Psi}_0(\mathbf{k}) 
               +\left(\Pi_{\widetilde{\mathcal{M}}_{+}^{(o)}}+(-1)^n\Pi_{\widetilde{\mathcal{M}}_-^{(o)}}\right)\widetilde{\Psi}_0(\mathbf{k}). 
\end{multline}
Define $\Psi_n(\mathbf{x})=\sum_{u\in V^{(o)}}\Psi_n(\mathbf{x},u)$. 
Applying $\mathcal{F}$ on both sides of Eq.~(\ref{yume}), we get 
\begin{multline}\label{fat}
	\Psi_n(\mathbf{x}) =  
        	\sum_{\phi\in \Phi^{(o)}\smallsetminus \Lambda}\int_{0\;d}^{2\pi} \exp\left\{in (\phi(\mathbf{k})-\langle \mathbf{k},\mathbf{x}/n \rangle)\right\} \Pi_{\omega_\phi}\widetilde{\Psi}_0(\mathbf{k}) \frac{d\mathbf{k}}{(2\pi)^{d}} \\
               +\sum_{\alpha \in \Lambda}e^{in \alpha} \int_{0\;d}^{2\pi} e^{-i\langle \mathbf{k},\mathbf{x} \rangle} \Pi_{\omega_\alpha}\widetilde{\Psi}_0(\mathbf{k}) \frac{d\mathbf{k}}{(2\pi)^{d}} \\
               +\int_{0\;d}^{2\pi} e^{-i\langle \mathbf{k},\mathbf{x} \rangle} \left(\Pi_{\widetilde{\mathcal{M}}_{+}^{(o)}}+(-1)^n\Pi_{\widetilde{\mathcal{M}}_-^{(o)}}\right)\widetilde{\Psi}_0(\mathbf{k})\frac{d\mathbf{k}}{(2\pi)^d}. 
\end{multline}
By Lemma~\ref{stationary}, \segawa{the first term vanishes in the limit of large $n$.} 
Then for large time step $n$, we have 
\begin{multline}
	P(X_n^{(\Psi_0)}=\mathbf{x}) \sim 
	\bigg|\bigg| \int_{0\;d}^{2\pi} e^{-i\langle \mathbf{k},\mathbf{x} \rangle} 
        \bigg\{ \sum_{\alpha\in \Lambda} e^{in \alpha}\Pi_{\omega_\alpha} 
        	+\Pi_{\widetilde{\mathcal{M}}_{+}^{(o)}}+(-1)^n\Pi_{\widetilde{\mathcal{M}}_-^{(o)}} \bigg\} \widetilde{\Psi}_0(\mathbf{k})\frac{d\mathbf{k}}{(2\pi)^d} \bigg|\bigg|^2.
\end{multline}
\segawa{Thus localization is assured by \segawa{the} appropriate initial state $\Psi_0$ satisfying $\widetilde{\Psi}_0(\mathbf{k}) \notin 
(\widetilde{\mathcal{M}}_{+}^{(o)}+\widetilde{\mathcal{M}}_{-}^{(o)}+\sum_{\alpha\in \Lambda}\mathbb{C}\omega_\alpha)^{\bot}$. }
In particular, \HS{$\tilde{\mathcal{M}}_\pm^{(o)} \neq \emptyset$ if $|E^{(o)}|-|V^{(o)}|\geq 1$; thus}
Proposition~\ref{fouriecycle} implies that 
if we choose the initial state $\Psi_0\in\ell^2(L\times D^{(o)})$ so that $\Psi_0\notin \mathrm{span}\{\mathcal{C}_+,\mathcal{C}_-\}^{\bot}$, 
then \[\limsup_{n\to\infty}P(X_n^{(\Psi_0)}=j)>0. \]
Thus we obtain the desired conclusion. 
\end{proof}
\begin{remark}
\segawa{Figure 1 shows a crystal lattice $|E^{(o)}|-|V^{(o)}|=0$ but localization occurs. 
This lattice can be also seen in Sect.~2.2~\cite{HS}. The details of DTQW on such a graph will be seen in the forth coming paper~\cite{HSeg}. }
\end{remark}
\begin{figure}[htbp]
\begin{center}
	\includegraphics[width=50mm]{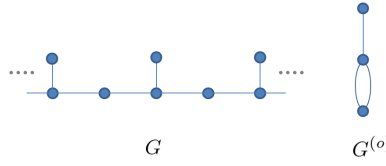}
\end{center}
\caption{{\scriptsize An example of a graph in which $|E^{(o)}|-|V^{(o)}|=0$ but $T^{(o)}_{\mathbf{k}}$ has a constant eigenvalue. 
$G$ is the original graph and $G^{(o)}$ is its quotient graph. }}
\label{fig:mushimegane}
\end{figure}
Finally we consider the contribution of $\Phi^{(o)}\smallsetminus \Lambda$ to the behavior of the Grover walk. 
For this purpose, we take the \segawa{scaling} to $X_n$; $X_n/n$. 
\begin{proposition}\label{weaklem}
Let the initial state be $\Psi_0\in \mathcal{H}$ and its Fourier transform be $\widetilde{\Psi}_0$. Then under assumptions (a) and (b), we have 
\begin{equation}\label{G-lemma}
\lim_{n\to \infty} E[e^{i\langle\bs{\xi}, X_n\rangle/n}]
	=c_0
        	+\int_{\mathbf{k}\in [0,2\pi)^d} \sum_{\phi\in \Phi_o}
        		e^{i\langle\bs{\xi},\nabla \phi\rangle}|| \Pi_{\omega_\phi}\widetilde{\Psi}_0(\mathbf{k}) ||^2  \frac{d\mathbf{k}}{(2\pi)^d}.
\end{equation}
Here 
\[c_0=  
	\int_{\mathbf{k}\in[0,2\pi)^d} \sum_{\alpha\in \Lambda} \left|\left|\Pi_{\omega_\alpha}\widetilde{\Psi}_0(\mathbf{k})\right|\right|^2 \frac{d\mathbf{k}}{(2\pi)^d}
        	+\left|\left|\left(\Pi_{\mathcal{C}_+}+\Pi_{\mathcal{C}_-}\right) \Psi_0\right|\right|^2. \]
\end{proposition}
\begin{proof}
From the definitions of the spatial Fourier transform and the characteristic function for $X_n$, we can write 
\begin{equation}\label{slemma} 
E[e^{i\langle\bs{\xi}, X_n\rangle}]=\int_{\mathbf{k}\in [0,2\pi)^d} \langle \widetilde{\Psi}_n(\mathbf{k}),\widetilde{\Psi}_n(\mathbf{k}+\bs{\xi}) \rangle \frac{d\mathbf{k}}{(2\pi)^d}. 
\end{equation}
\HS{Replacing $\bs{\xi}$ with $\bs{\xi}/n$, we have the expressions $\phi\in \Psi_o$, $\phi(\mathbf{k}+\bs{\xi}/n)=\phi(\mathbf{k})+\langle \bs{\xi}/n, \nabla \phi(\mathbf{k})\rangle+O(1/n^2)$, 
and the eigenfunctions $\omega_\phi(\mathbf{k}+\bs{\xi}/n,e)=\omega_\phi(\mathbf{k},e)+O(1/n)$. 
Substituting these and Eq.~(\ref{yume}) with $\widetilde{\Psi}_n(\cdot)$ in Eq.~(\ref{slemma}), 
we can find that all cross terms in the inner product of the integrand in Eq.~(\ref{slemma}) are $O(1/n)$ by Lemma~\ref{stationary}. 
Thus if $n\to\infty$, only the diagonal terms retained as follows: }
\[ \sum_{\phi\in \Phi_o}
        		e^{i\langle\bs{\xi},\nabla \phi\rangle}|| \Pi_{\omega_\phi}\widetilde{\Psi}_0(\mathbf{k}) ||^2
                        +|| \Pi_{\tilde{\mathcal{M}}^{(o)}_+}\widetilde{\Psi}_0(\mathbf{k}) ||^2+|| \Pi_{\tilde{\mathcal{M}}^{(o)}_-}\widetilde{\Psi}_0(\mathbf{k}) ||^2
                        +\sum_{\alpha\in \Lambda} ||\Pi_{\omega_\alpha}\widetilde{\Psi}_0(\mathbf{k})||^2. \] 
The integral of the second term $|| \Pi_{\tilde{\mathcal{M}}^{(o)}_+}\widetilde{\Psi}_0(\mathbf{k}) ||^2$ is rewritten as 
\begin{align*}
	\int_{0\;d}^{2\pi} || \Pi_{\tilde{\mathcal{M}}^{(o)}_+}\widetilde{\Psi}_0(\mathbf{k}) ||^2 \frac{d\mathbf{k}}{(2\pi)^d}
        	&= \sum_{\mathbf{x}\in L} \left|\left|\int_{0\;d}^{2\pi} \Pi_{\tilde{\mathcal{M}}^{(o)}_+}\widetilde{\Psi}_0(\mathbf{k}) e^{-i\langle \mathbf{k},\mathbf{x} \rangle}\frac{d\mathbf{k}}{(2\pi)^d} \right|\right|^2 \\
                &= \sum_{\mathbf{x}\in L} \left|\left|\Pi_{\mathbf{x}}\Pi_{\mathcal{C}_+}\Psi_0\right|\right|^2 \\
                &= \left|\left|\Pi_{\mathcal{C}_+}\Psi_0\right|\right|^2. 
\end{align*}
Here the second equality follows from Proposition~\ref{fouriecycle}. 
\segawa{Similarly, the third term is integrated as} 
\[ \int_{0\;d}^{2\pi} || \Pi_{\tilde{\mathcal{M}}^{(o)}_-}\widetilde{\Psi}_0(\mathbf{k}) ||^2 \frac{d\mathbf{k}}{(2\pi)^d}=\left|\left|\Pi_{\mathcal{C}_-}\Psi_0\right|\right|^2. \]
Thus we complete the proof. 
\end{proof}

\begin{theorem}\label{linear_spreading}
Under assumptions (a) and (b), the Grover walk on the crystal lattice $G$ exhibits linear spreading. 
\end{theorem}
\begin{proof}
\segawa{The proof is directly follows from Proposition~\ref{weaklem}.}  
\end{proof}
\section{\segawa{Examples}}
\segawa{
In this section, we confirm that the Grover walks on the $d$-dimensional square lattice, triangular lattice and hexagonal lattice satisfy Assumption~\ref{assumption};
that is, these walks exhibit both localization and linear spreading. 
Finally we compute explicit expressions for the proper limit distribution 
of the Grover walk on the $d$-dimensional square lattice, $\mathbb{Z}^d$. 
}

Given \segawa{a} crystal lattice $G$ with the quotient graph $G^{(o)}$, we assume that the lattice $L$ is embedded in $\mathbb{R}^d$. 
The $r$-unit \segawa{vector} $\hat{\theta}_1,\dots,\hat{\theta}_r\in \mathbb{R}^d$ with $r\geq d$ corresponds to $1$-form 
$\theta(e_1),\dots,\theta(e_r)$, $(e_j\in E^{(o)}\smallsetminus E(\mathbb{T}^{(o)}))$; that is, $\theta(e_j)=\langle \hat{\theta}_j,\mathbf{k} \rangle$ for a fixed $\mathbf{k}={}^T(k_1,\dots,k_d)\in K^d$. 
Here $r=|E^{(o)}|-|E(\mathbb{T}^{(o)}))|$. 
From $\{\hat{\theta}_1,\dots,\hat{\theta}_r\}$, we can choose $d$-linearly independent vectors $\{\hat{\theta}_{n_1},\dots,\hat{\theta}_{n_d} \}$. 
Define \segawa{the} $d\times d$ matrix $J$ by $J={}^T[\hat{\theta}_{n_1},\dots,\hat{\theta}_{n_d}]$. 
The linearly independency implies $\mathrm{det}(J)\neq 0$. 
Taking $\theta_j\equiv \theta(e_{n_j})$, 
remark that  
$\partial \phi/\partial k_1=\partial \phi/\partial k_2=\cdots=\partial \phi/\partial k_d=0$ if and only if 
$\partial \phi/\partial \theta_{1}=\partial \phi/\partial \theta_{2}=\cdots=\partial \phi/\partial \theta_{d}=0$. 
Moreover since $J\widetilde{\mathrm{H}}_\theta {}^{T}J=\mathrm{H}_k$, 
it holds that \[\mathrm{det}(\mathrm{H}_k)=\mathrm{det}^2(\mathrm{J})\mathrm{det}(\widetilde{\mathrm{H}}_\theta). \] 
Here 
\[ (\mathrm{H}_k)_{l,m}=\frac{\partial^2\phi}{\partial k_l\partial k_m} \mathrm{\;and\;} 
(\widetilde{\mathrm{H}}_\theta)_{l,m}=\frac{\partial^2\phi}{\partial \theta_l\partial \theta_m}. \]
\segawa{So we use parameters $\theta_1,\dots,\theta_d$ instead of $k_1,\dots,k_d$ in the following three examples. }
\begin{figure}[htbp]
\begin{center}
	\includegraphics[width=100mm]{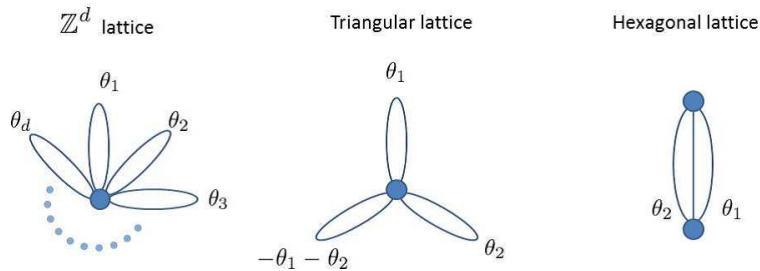}
\end{center}
\caption{The quotient graphs of $\mathbb{Z}^d$, triangular lattice and hexagonal lattice}
\label{fig:one}
\end{figure}
\subsection{\HS{Square lattices}}
\segawa{The $d$-dimensional square lattice is well-known as} the maximal abelian covering graph of the $d$-bouquet; 
that is, one vertex with $d$ self loops $\{e_1,\dots,e_d\}$. 
The $1$-form assigns $\theta(e_1)=\theta_1,\dots,\theta(e_d)=\theta_d$. 

\segawa{The characteristic function of the simple random walk on $\mathbb{Z}^d$ starting from the origin is described by 
\[ \widetilde{P}^{(o)}_\theta=\left( \frac{ \sum_{j=1}^d e^{i\theta_j}+e^{-i\theta_j}} {2d}\right)^n=\left(\frac{\sum_{j=1}^d \cos \theta_j}{d}\right)^n \]
Thus we have 
	\begin{equation} \cos\phi= 1/d\cdot \sum_{j=1}^d \cos \theta_j\end{equation} }
which implies 
\begin{align}\label{zbibun}
	\frac{\partial \phi}{\partial \theta_j} = \frac{\sin\theta_j}{d\sin \phi}. 
\end{align}
Then all candidates of the critical points, in which the numerators of RHS in Eq.~(\ref{zbibun}) are zero for all $j$, are given by:
\[ \{(n_1\pi,\dots,n_d\pi): n_j\in\{0,1\} \}. \]
\segawa{The case $(\theta_1,\dots,\theta_d)\in \{(0,\dots,0),(\pi,\dots,\pi)\}$ requires special attention} because 
\segawa{the denominators of RHS in Eq.~(\ref{zbibun}) at these points are also zero.} 
Taking $\theta_j=\theta_j(\epsilon)$ with $\lim_{\epsilon \to 0}\theta_j(\epsilon)=0$, 
we have $\cos\theta\sim 1-1/2\cdot \theta_j^2(\epsilon)$ and $\sin\phi\sim  \sqrt{K_2(\epsilon)/d}$, 
where $K_2(\epsilon)=\sum_{j=1}^d \theta_j^2$. Inserting these approximations into Eq.~(\ref{zbibun}), we obtain 
\begin{equation}\label{popopo} 
\frac{\partial \phi}{\partial \theta_j}\sim \frac{1}{\sqrt{d}}\frac{\theta_j(\epsilon)}{\sqrt{K_2(\epsilon)}}.  
\end{equation}
\segawa{Here $f(x)\sim g(x)$ indicates $\lim_{x\to 0} |f(x)/g(x)|=1$.} 
Putting $t_j=\lim_{\epsilon \to 0} \theta_j(\epsilon)/\sqrt{K_2(\epsilon)}$, 
Eq.~(\ref{popopo}) implies that 
\[ (t_1,\dots,t_d)\in \partial B_d, \]  
where $\partial B_d=\{\mathbf{x}\in \mathbb{R}^d: ||\mathbf{x}||^2=1/d\}$. 
Obviously $\segawa{(0,\dots,0)}$ is outside of $\mathcal{S}_\phi$. 
In the same way, we obtain $(\pi,\dots,\pi)\notin \mathcal{S}_\phi$. 
Thus we have $\mathcal{S}_\phi=\{(n_1\pi,\dots,n_d\pi): n_j\in\{0,1\} \}\smallsetminus \{(0,\dots,0),(\pi,\dots,\pi)\}$. 

The Hessian is given by Eq.~(\ref{detJ1}). 
For $\mathbf{p}=(k_1,\dots,k_d)\in \mathcal{S}_\phi$, 
$\cos \phi(\mathbf{p})=1-2m/d$, where $m$ is the number of $\segawa{\pi}$ in the sequence $(k_1,\dots,k_d)$. 
Then it is easily checked that for all $\mathbf{p}\in \mathcal{S}_\phi$, 
\[ |\mathrm{det}(\mathrm{H}(\mathbf{p}))|=\segawa{\left(\frac{1}{2\sqrt{m(d-m)}}\right)^d} \neq 0. \]
\subsection {Triangular lattice} 
Consider the quotient graph $G^{(o)}$; one vertex with three self loops $\{e_1,e_2,e_3\}$. 
The $1$-form is $\theta(e_1)=\theta_1$, $\theta(e_2)=\theta_2$ and $\theta(e_3)=\theta_3$. 
\segawa{The triangular lattice is the abelian covering graph of $G^{(o)}$ under the relation $\theta_1+\theta_2+\theta_3=0$}. 
The transition matrix for the twisted walk on the quotient graph is described by 
\[ \widetilde{P}^{(o)}_{\segawa{\theta}}=\frac{1}{3}\left( \cos\theta_1+\cos\theta_2+\cos(\theta_1+\theta_2) \right). \]
The critical points are obtained by computing 
\begin{align*}
	\frac{\partial \phi}{\partial \theta_1} &= \frac{\sin\theta_1+\sin(\theta_1+\theta_2)}{3\sin \phi}=0, \\
	\frac{\partial \phi}{\partial \theta_2} &= \frac{\sin\theta_2+\sin(\theta_1+\theta_2)}{3\sin \phi}=0,
\end{align*}
implying that 
\[ \mathcal{S}_\phi=\{(0,\pi),(\pi,0),(\pm 2\pi/3,\pm 2\pi/3),(\pi,\pi)\}. \]
\segawa{We remark that numerators and denominators of $\partial \phi/\partial \theta_1$ and $\partial \phi/\partial \theta_2$ 
at the origin $(\theta_1,\theta_2)=(0,0)$ are both zero.} 
We now show that $(\theta_1,\theta_2)=(0,0)$ is outside of $\mathcal{S}_\phi$.
To this end, we take the limit close to the origin using a single parameter $\epsilon\in \mathbb{R}$, and set 
$\theta_1=\theta_1(\epsilon)$ and $\theta_2=\theta_2(\epsilon)$ with $\lim_{\epsilon\to 0}\theta_j(\epsilon)=0$. 
When $\epsilon$ is small, we can adopt the asymptotic forms of $\cos \theta_j\sim 1-\theta_j^2/2$ and $\sin \theta_j\sim \theta_j$ to obtain 
\begin{align}
\frac{\partial \phi}{\partial \theta_1}
	&\sim \frac{\alpha(\epsilon)+2}{\sqrt{6(\alpha^2(\epsilon)+\alpha(\epsilon)+1)}}, \\
\frac{\partial \phi}{\partial \theta_2}
	&\sim \frac{2\alpha(\epsilon)+1}{\sqrt{6(\alpha^2(\epsilon)+\alpha(\epsilon)+1)}}.        
\end{align}
Here $\alpha(\epsilon)=\theta_2(\epsilon)/\theta_1(\epsilon)$. 
We put \[ x=\frac{t+2}{\sqrt{6(t^2+t+1)}},\;y=\frac{2t+1}{\sqrt{6(t^2+t+1)}}. \]
Rotating $x$ and $y$ by one-quarter turn, we get \[ x'=\frac{1}{\sqrt{2}}(x+y),\;y'=\frac{1}{\sqrt{2}}(-x+y). \]
Then we obtain the equation of an ellipse: 
\[ (x')^2+\frac{(y')^2}{\left(\frac{1}{3}\right)}=1. \]
It holds that 
\[ (0,0)\notin \left\{\left(\frac{t+2}{\sqrt{6(t^2+t+1)}},\;\frac{2t+1}{\sqrt{6(t^2+t+1)}}\right): t\in \mathbb{R}  \right\}. \]
Therefore $(0,0)$ is outside of $\mathcal{S}_\phi$. 

The Hessian matrix is 
\[ \mathrm{H}_\theta|_{\mathcal{S}_\phi}
	=\frac{1}{\sin \phi}
        \begin{bmatrix} \cos\theta_1+\cos(\theta_1+\theta_2) & \cos(\theta_1+\theta_2) \\ 
                        \cos(\theta_1+\theta_2) & \cos(\theta_1+\theta_2)+\cos\theta_2 
        \end{bmatrix} 
\]
It is easily checked that for all $p\in \mathcal{S}_\phi$, $\mathrm{det}(\mathrm{H}_\theta(p))\neq 0$. 
\subsection{Hexagonal lattice } 
The hexagonal lattice is the maximal abelian covering graph of $G^{(o)}=(V^{(o)},D^{(o)})$ with 
\[ V^{(o)}=\{u,v\},\;\; D^{(o)}=\{e_0,e_1,e_2,\bar{e}_0,\bar{e}_1,\bar{e}_2\}, \] 
where $u=o(e_0)=o(e_1)=o(e_2)$ and $v=t(e_0)=t(e_1)=t(e_2)$. 
The $1$-form assigns $\theta(e_0)=0$, $\theta(e_1)=\theta_1$ and $\theta(e_2)=\theta_2$. 
The transition matrix of the twisted walk on the quotient graph is described by 
\[ \widetilde{P}^{(o)}_k=\frac{1}{3}
	\begin{bmatrix} 
        0& 1+e^{-i\theta_1}+e^{-i\theta_2} \\ 1+e^{i\theta_1}+e^{i\theta_2} & 0 
        \end{bmatrix}.
\]
Its eigenvalues are $\{\pm |1+e^{i\theta_1}+e^{i\theta_2}|/3 \}$. 
Taking $\cos \phi=|1+e^{i\theta_1}+e^{i\theta_2}|/3$, we have 
\begin{align}
	\frac{\partial \phi}{\partial \theta_1} &= \frac{2}{9}\frac{\sin\theta_1+\sin(\theta_1-\theta_2)}{\sin 2\phi}, \label{hex1}\\
	\frac{\partial \phi}{\partial \theta_2} &= \frac{2}{9}\frac{\sin\theta_2+\sin(\theta_2-\theta_1)}{\sin 2\phi}.\label{hex2}
\end{align}
The candidates of the critical points of $\phi$ are 
\[ \left\{ (0,0),(\pm 2\pi/3,\mp 2\pi/3),(\pm \pi,\mp \pi),(0,\pm\pi),(\pm\segawa{\pi},0) \right\} \]
\segawa{The above are solution for which both numerators of RHSs in Eqs.~(\ref{hex1}) and (\ref{hex2}) are equal to zero. }
We now show that the first two candidates $(0,0)$ and $(\pm 2\pi/3,\mp 2\pi/3)$ are excluded as the critical points. 
We take also $\theta_j=\theta_j(\epsilon)$ with $\lim_{\epsilon\to 0}\theta_j(\epsilon)=0$ 
and $\alpha(\epsilon)=\theta_2(\epsilon)/\theta_1(\epsilon)$. 
\begin{enumerate}
\item $(0,0)$ case : We can evaluate $\sin\theta_j\sim \theta_j(\epsilon)$ and 
$\sin(\theta_1-\theta_2)\sim \theta_1(\epsilon)-\theta_2(\epsilon)$, moreover 
\begin{align} 
\cos 2\phi &= 2\cos^2\phi-1=-\frac{1}{3}+\frac{4}{9}\left(\cos\theta_1+\cos\theta_2+\cos(\theta_1-\theta_2)\right)  \\
	&\sim 1-\frac{4}{9}\left(\theta_1^2(\epsilon)+\theta_2^2(\epsilon)-\theta_1(\epsilon)\theta_2(\epsilon)\right). \\
\sin^2 2\phi 
	&\sim \frac{8}{9}\left(\theta_1^2(\epsilon)+\theta_2^2(\epsilon)-\theta_1(\epsilon)\theta_2(\epsilon)\right).       
\end{align}
Inserting these estimations into Eqs.~(\ref{hex1}) and (\ref{hex2}), we get  
\begin{align}
	\frac{\partial \phi}{\partial \theta_1} &= \frac{2-\alpha(\epsilon)}{3\sqrt{2}\sqrt{\alpha^2(\epsilon)-\alpha(\epsilon)+1}}, \label{hex3}\\
	\frac{\partial \phi}{\partial \theta_2} &= \frac{-1+2\alpha(\epsilon)}{3\sqrt{2}\sqrt{\alpha^2(\epsilon)-\alpha(\epsilon)+1}}.\label{hex4}
\end{align}
We put \[ x=\frac{2-t}{3\sqrt{2}\sqrt{t^2-t+1}},\;y=\frac{-1+2t}{3\sqrt{2}\sqrt{t^2-t+1}}, \]
Rotating $(x,y)$ by one-quarter turn gives \[ x'=\frac{1}{\sqrt{2}}(x+y),\;y'=\frac{1}{\sqrt{2}}(-x+y). \]
Then we have the following equation of ellipse: 
\[ \frac{(x')^2}{\left( \frac{1}{9} \right)}+\frac{(y')^2}{\left(\frac{1}{3}\right)}=1. \]
It holds that 
\[ (0,0)\notin \left\{\left(\frac{2-t}{3\sqrt{2}\sqrt{t^2-t+1}},\;\frac{-1+2t}{3\sqrt{2}\sqrt{t^2-t+1}}\right): t\in \mathbb{R}  \right\}. \]
Therefore $(0,0)$ is outside of $\mathcal{S}_\phi$. 
\item $(\pm 2\pi/3,\mp 2\pi/3)$ case : 
We consider $(2\pi/3,-2\pi/3)$ case. Set $\theta_1=2\pi/3+\eta_1(\epsilon)$, $\theta_2=-2\pi/3+\eta_2(\epsilon)$ with 
$\eta_j(\epsilon)\to 0$. We estimate 
\begin{align*} 
\sin \theta_j &\sim \frac{\sqrt{3}}{2}-\frac{1}{2}\eta_j(\epsilon) \\
|\sin 2\phi| &\sim \frac{2}{3}\sqrt{\eta_1^2(\epsilon)+\eta_2^2(\epsilon)-\eta_1(\epsilon)\eta_2(\epsilon)} 
\end{align*}
Inserting these estimations into Eqs.~(\ref{hex1}) and (\ref{hex2}), 
\begin{align}
	\frac{\partial \phi}{\partial \theta_1} &= \frac{2-\beta(\epsilon)}{6\sqrt{\beta^2(\epsilon)-\beta(\epsilon)+1}}, \label{hex3}\\
	\frac{\partial \phi}{\partial \theta_2} &= \frac{2\beta(\epsilon)-1}{6\sqrt{\alpha^2(\epsilon)-\alpha(\epsilon)+1}}.\label{hex4}
\end{align}
We put \[ x=\frac{2-t}{6\sqrt{t^2-t+1}},\;y=\frac{-1+2t}{6\sqrt{t^2-t+1}}, \]
and rotate it a quarter turn \[ x'=\frac{1}{\sqrt{2}}(x+y),\;y'=\frac{1}{\sqrt{2}}(-x+y). \]
Then we have the following equation of ellipse: 
\[ \frac{(x')^2}{\left( \frac{1}{18} \right)}+\frac{(y')^2}{\left(\frac{1}{6}\right)}=1. \]
It holds that 
\[ (0,0)\notin \left\{\left(\frac{2-t}{3\sqrt{2}\sqrt{t^2-t+1}},\;\frac{-1+2t}{3\sqrt{2}\sqrt{t^2-t+1}}\right): t\in \mathbb{R}  \right\}. \]
Therefore $(0,0)$ is outside of $\mathcal{S}_\phi$. 
\end{enumerate}

The Hessian matrix is 
\[ \mathrm{H}_\theta|_{\mathcal{S}_\phi}
	=\segawa{\frac{2}{9\sin 2\phi}}
        \begin{bmatrix} \cos\theta_1+\cos(\theta_1-\theta_2) & -\cos(\theta_1-\theta_2) \\ 
                        -\cos(\theta_1-\theta_2) & \cos(\theta_1-\theta_2)+\cos\theta_2 
        \end{bmatrix} 
\]
It is easily checked that for all $p\in \mathcal{S}_\phi$, $\mathrm{det}(\mathrm{H}_\theta(p))\neq 0$. 
\subsection{\segawa{Limit measures of the Grover walk on $\mathbb{Z}^d$ lattice}} 
\segawa{In this section, we discuss more explicit expression for the limit behaviors which is proper to the Grover walk on $\mathbb{Z}^d$. 
To do so, we introduce two limit measures for \segawa{localization} and linear spreading of the Grover walk on 
$\mathbb{Z}^d$. } 
\segawa{
\begin{definition} ~\cite{KLS}
\noindent
\begin{enumerate}
\item We define the time averaged limit measure $\overline{\mu}_\infty : \mathbb{Z}^d\to [0,1]$ as 
	\[ \overline{\mu}_\infty(\mathbf{x})=\lim_{T\to\infty}\frac{1}{T}\sum_{n=0}^{T-1}\mu_n(\mathbf{x}). \]
\item We define the weak limit measure $d\mu_w^{(d)}: \mathbb{R}^d \to \mathbb{R}^d$ as 
	\[ \lim_{n\to\infty}P(X_n/n<\mathbf{y})=\int_{\{\mathbf{x}\in\mathbb{R}^d:\mathbf{x}<\mathbf{y}\}} d\mu_w^{(d)}(\mathbf{x}). \]
Here for $\mathbf{x}=(x_1,\dots,x_d),\;\mathbf{y}=(y_1,\dots,y_d)\in \mathbb{R}^d$, $\mathbf{x}<\mathbf{y}$ denotes $x_1<y_1,\;\dots,x_d<y_d$.
We call $d\mu_w^{(d)}$ a weak limit measure. 
\end{enumerate}
\end{definition}
}
\segawa{We remark that $\overline{\mu}_\infty(\mathbf{x})>0$ implies localization and 
$d\mu_w^{(d)}\neq \delta_\mathbf{o}$ implies linear spreading. } 
\subsubsection{\segawa{Time averaged limit measure}}
We express for $\hat{\theta}_j\in \mathbb{R}^d$ by   
\[ \hat{\theta}_1={}^T[1,0,\dots,0], \; \hat{\theta}_2={}^T[0,1,\dots,0],\dots,\hat{\theta}_d={}^T[0,0,\dots,1]. \]
Since $|V^{(o)}|=|\{u\}|=1$, we denote $\Psi_n(\mathbf{x},u)=\Psi_n(\mathbf{x})$ for $\mathbf{x}\in \mathbb{Z}^d$ and $\widehat{\Psi}_n(\mathbf{k},u)=\widehat{\Psi}_n(\mathbf{k})$ for $\mathbf{k}\in \mathbb{R}^d$. 
\HS{Let $G_d$ be the $d$-dimensional Grover matrix $2/d\cdot J_d-I_d$, where $J_d$ is the $d$-dimensional all $1$ matrix. }
Equation (\ref{FT}) implies 
\[ \widetilde{\Psi}_n(\mathbf{k})=\widetilde{U}^{(o)}(\mathbf{k})^n \widetilde{\Psi}_0(\mathbf{k}). \]
Here $\widetilde{U}^{(o)}(\mathbf{k})=\widetilde{S}^{(o)}(\mathbf{k})G_d$, where 
for $\mathbf{k}=(k_1,k_2,\dots,k_d)$, $l\in\{\pm 1,\pm 2,\dots,\pm d\}$, 
$\widetilde{S}^{(o)}(\mathbf{k})\delta_{e_l}=e^{ik_l}\delta_{\bar{e}_{l}}$. 

As seen in the previous subsection, $\mathbb{Z}^d$ satisfies \ref{assumption}. 
By Theorem~\ref{loc_cycle}, we have the following corollary. 
\segawa{This corollary geometrically expresses the eigenspaces $\mathcal{C}_\pm$ that are proper to $\mathbb{Z}^d$. }
\begin{corollary}\label{z^dloc}
Let $\bs{\gamma}$ and $\bs{\tau}$ be defined in Eqs~(\ref{wawa}) and (\ref{vava}), respectively. 
For $\mathbf{x}\in \mathbb{Z}^d$ and $e,f\in D^{(o)}$, 
define a closed path $c_{\mathbf{x};e,f}$ on $\mathbb{Z}^d\times D^{(o)}$ by 
\[ \left((\mathbf{x},e),(\mathbf{x}+\hat{\theta}(e),f),(\mathbf{x}+\hat{\theta}(e)+\hat{\theta}(f),\bar{e}),(\mathbf{x}+\hat{\theta}(f),\bar{f})\right). \]
Then we have 
\begin{align}
\mathcal{C}_+ &= \sum_{\mathbf{x}\in \mathbb{Z}^d,\;e,f\in D^{(o)}} \mathbb{C}\bs{\gamma}(c_{\mathbf{x};e,f}),\\  
\mathcal{C}_- &= \sum_{\mathbf{x}\in \mathbb{Z}^d,\;e,f\in D^{(o)}} \mathbb{C}\bs{\tau}(c_{\mathbf{x};e,f}). 
\end{align}
\segawa{Moreover the time averaged limit measure, given an initial state $\Psi_0\in \mathcal{H}$, 
is expressed by }
\begin{equation}
\segawa{\overline{\mu}_\infty(j)
=\sum_{e\in D^{(o)}}\left|\left(\left(\Pi_{\mathcal{C}_+}+ \Pi_{\mathcal{C}_-}\right)\Psi_0\right)(j,e)\right|^2. }
\end{equation}
\end{corollary}

\subsubsection{\segawa{Weak limit measure}}
In this subsection, we simplify the problem by adopting a mixed state as the initial state; that is, 
\[ P[\Psi_0=\bs{\delta_{e}}]= \frac{\bs{1}_{\{o(e)=0\}}(e)}{2d}.  \]
We now consider the weak limit theorems by taking the expectation with respect to the initial state. 
In the mixed state, Proposition~\ref{weaklem} reduces to the following corollary. 
\begin{corollary}
For a mixed initial state, we have 
\begin{equation}\label{G-lemma2}
\lim_{n\to \infty} E[e^{i\langle\bs{\xi}, X_n\rangle/n}]
	=\left(1-\frac{1}{d}\right)
        	+\frac{1}{d}\int_{\mathbf{k}\in [0,2\pi)^d} \cos [\langle\bs{\xi},\nabla \phi\rangle] \frac{d\mathbf{k}}{(2\pi)^d}.
\end{equation}
\end{corollary}
\begin{proof}
The first term of RHS in Eq.~(\ref{G-lemma}) is rewritten as 
\begin{equation}\label{MI} 
|| (\Pi_{\mathcal{C}_+}+\Pi_{\mathcal{C}_-})\Psi_0 ||^2
	=\int_{0\;d}^{2\pi}|| (\Pi_{\widetilde{\mathcal{M}}_+^{(o)}}+\Pi_{\widetilde{\mathcal{M}}_-^{(o)}})\widetilde{\Psi}_0(\mathbf{k}) ||^2 \frac{d\mathbf{k}}{(2\pi)^d}. 
\end{equation}  
Since \segawa{$Prob(``\widetilde{\Psi}_0=\delta_{e_0}")=1/(2d)$}, the integrand of RHS in Eq.~(\ref{MI}) reduces to 
\begin{align*} 
|| (\Pi_{\widetilde{\mathcal{M}}_+^{(o)}}+\Pi_{\widetilde{\mathcal{M}}_-^{(o)}})\widetilde{\Psi}_0(\mathbf{k}) ||^2
	&= \frac{1}{2d}\mathrm{Tr}\left[ \Pi_{\widetilde{\mathcal{M}}_+^{(o)}}+\Pi_{\widetilde{\mathcal{M}}_-^{(o)}} \right] \\
        &= \frac{1}{2d} \left(\mathrm{dim}(\widetilde{\mathcal{M}}_+^{(o)})+\mathrm{dim}(\widetilde{\mathcal{M}}_-^{(o)})\right) \\
        &= \frac{d-1}{d} \;\;\;\;\;\; (a.s.)
\end{align*}
Moreover, since  
\sato{\begin{equation}\label{MII} 
|| \Pi_{w_\phi}\widetilde{\Psi}_0(\mathbf{k}) ||^2
	= \frac{1}{2d}\mathrm{Tr}[\Pi_{w_\phi}]=\frac{1}{2d}, 
\end{equation}  }
the integrand of the second term of RHS in Eq.~(\ref{G-lemma}) becomes 
\[ \sum_{\phi} e^{i\langle\bs{\xi},\nabla \phi\rangle}|| \Pi_{w_\phi}\widetilde{\Psi}_0(\mathbf{k}) ||^2  
	= \frac{1}{2d} \left(e^{i\langle\bs{\xi},\nabla \phi(\mathbf{k})\rangle}+e^{-i\langle\bs{\xi},\nabla \phi(\mathbf{k})\rangle}\right). \]
Thereby, we complete the proof. 
\end{proof}
For $\phi\in \Phi^{(o)}\smallsetminus \Lambda$ with $\cos\phi(\mathbf{k})=(\sum_{j=1}^d\cos k_j)/d$ and $\sin\phi(\mathbf{k})>0$, 
we define the Hessian matrix $\mathrm{Hess}_\phi=(\mathrm{H}_d)_{l,m}$ $(l,m\in{1,\dots,d})$. 
\begin{lemma}\label{detJ}
Taking $x_j=\partial \phi(\mathbf{k})/\partial k_j$ with $\cos\phi(\mathbf{k})=(\sum_{j=1}^d\cos k_j)/d$ and $\sin\phi(\mathbf{k})\geq 0$, 
we have 
\begin{equation}\label{detJ1}
\segawa{\mathrm{det}(\mathrm{H}_d)}
	=\prod_{l=1}^d\left(\frac{\cos k_l}{d\sin\phi(\mathbf{k})}\right)\left( 1-d\cos\phi(\mathbf{k}) \sum_{j=1}^d \frac{x_j^2}{\cos k_j}\right).
\end{equation}
\end{lemma}
\begin{proof}
Since $\cos\phi(\mathbf{k})=1/d\cdot \sum_{j=1}^d \cos k_j$, we have 
\begin{equation}\label{critical}
\frac{\partial \phi(\mathbf{k})}{\partial k_m} = \frac{1}{d}\frac{\sin k_m}{\sin \phi(\mathbf{k})}. 
\end{equation}
Moreover remarking 
\begin{equation*}
x_m=\frac{\partial}{\partial k_m}\phi(\mathbf{k})=\frac{1}{d} \frac{\sin k_m}{\sin \phi(\mathbf{k})},
\end{equation*}
we obtain 
\begin{equation*} 
\frac{\partial}{\partial k_l}\frac{\partial}{\partial k_m}\phi(\mathbf{k})
	= -\frac{1}{\tan \phi(\mathbf{k})}
        \left( x_lx_m-\delta_{l,m} \frac{\cos k_l}{d\cos \phi(\mathbf{k})} \right). 
\end{equation*}
Thus $\mathrm{H}_d$ is expressed as follows: 
\begin{equation*}
\mathrm{H}_d = - \frac{1}{\tan\phi(\mathbf{k})}(P-D), 
\end{equation*}
where $(P)_{l,m}=x_lx_m$, and $D=\mathrm{diag}[\cos k_j/(d\cos\phi(\mathbf{k})); 1\leq j\leq d]$. 
The determinant of $\mathrm{H}_d$ is computed as 
\begin{align}
\mathrm{det}(\mathrm{H}_d) &= \left(-\frac{1}{\tan\phi(\mathbf{k})}\right)^d\mathrm{det}(P-D)  \notag \\ 
	&= \left(-\frac{1}{\tan\phi(\mathbf{k})}\right)^d\mathrm{det}(D)\mathrm{det}(I-D^{-1}P). \label{JJ}
\end{align}
We remark that 
\begin{align}
	\mathrm{det}(D) &= \prod_{j=1}^d \frac{\cos k_j}{d\cos\phi(\mathbf{k})} \label{DD}\\
        \mathrm{det}(I-D^{-1}P) &= 1-\mathrm{Tr}[D^{-1}P] \notag \\ 
        			&= 1-\sum_{j=1}^d \frac{d\cos\phi(\mathbf{k})}{\cos k_j}x_j^2 \label{PP}
\end{align}
Here we used the fact that $\mathrm{det}(I_n-AB)=\mathrm{det}(I_m-BA)$, where $A$ and $B$ are $n\times m$ and $m\times n$ matrices, respectively. 
Inserting Eqs.~(\ref{DD}) and (\ref{PP}) into Eq.~(\ref{JJ}) completes the proof.
\end{proof}

We define the density function of the weak convergence of QW with respect to RHS of the second term in Eq.~(\ref{G-lemma}) by $\rho_d(\mathbf{x})$. 
\segawa{From Eq.~(\ref{G-lemma2}), the weak limit measure is 
	\[d\mu_w^{(d)}(\mathbf{x})=(1-1/d)\delta_\mathbf{0}(\mathbf{x})+\rho_d(\mathbf{x}), \] 
where $\int_{-\infty\;d}^\infty\; \rho_d(\mathbf{x})d\mathbf{x}=1/d$.  }
Replacing $\partial \phi(\mathbf{k})/\partial k_j$ with $x_j$ in the integral of RHS in Eq.~(\ref{G-lemma}), 
we can find the shape of \segawa{the continuous part of} the limit density function $\rho_d$ from the Hessian matrix $\mathrm{H}_d$; that is, for $\mathbf{x}=(x_1,\dots,x_d)$, 
\begin{equation}\label{sanriku} 
\rho_d(\mathbf{x})\propto \frac{1}{(2\pi)^d|\mathrm{det}(\mathrm{H}_d)|}. 
\end{equation} 
\begin{theorem}\label{limdis2}
When $d=2$, 
\begin{equation}\label{lim2}
\rho_2(x,y)=\frac{2\times \bs{1}_{\{x^2+y^2\leq 1/2\}}(\mathbf{x})}{\pi^2 (x+y-1)(x+y+1)(x-y+1)(x-y-1)}. 
\end{equation}
\end{theorem}
\begin{proof}
Explicit solutions $(k,l)\in K^2$ to 
\[x=\frac{\sin k}{2\sin\phi(k,l)}, \;\;y=\frac{\sin l}{2\sin\phi(k,l)}, \]
can be obtained as the function of $x$ and $y$: 
\begin{align}
\cos k &= \frac{1-3x^2-y^2}{\sqrt{(x+y-1)(x+y+1)(x-y-1)(x-y+1)}}, \label{cosk}\\
\cos l &= -\frac{1-3y^2-x^2}{\sqrt{(x+y-1)(x+y+1)(x-y-1)(x-y+1)}}. \label{cosl}
\end{align}
The proof is completed by inserting Eqs.~(\ref{cosk}) and (\ref{cosl}) into Lemma~\ref{detJ}. 
\end{proof}
\segawa{When $d\geq 3$, the following system of equations in terms of $k_1,k_2,\dots,k_d$ is generally difficult to solve directly}: 
\begin{align*} 
x_1 = \frac{\sin k_1}{d\sin\phi(\mathbf{k})},\; 
x_2 = \frac{\sin k_2}{d\sin\phi(\mathbf{k})},\; 
     \cdots,\; 
x_d =  \frac{\sin k_d}{d\sin\phi(\mathbf{k})}.  
\end{align*}
So it is hard to obtain a closed expression $\rho_d$ as a function of $\mathbf{x}\in\mathbb{R}^d$ at this stage. 
However, we can partially obtain the shape of $\rho_d$ as shown below. 
\begin{theorem}\label{ball}
Let the support of $\rho_d$ be $S_d\subset \mathbb{R}^d$. Then we have 
\begin{equation*}
S_d\subseteq B_d, 
\end{equation*}
where $B_d$ is the $d$-dimensional sphere of radius $1/\sqrt{d}$. 
\end{theorem}
\begin{proof}
\begin{align}
\sum_{j=1}^d x_j^2 &= \sum_{j=1}^d \left(\frac{1}{d} \frac{\sin k_j}{\sin \phi(\mathbf{k})}\right)^2
	            =\frac{\sum_{j=1}^d(1-\cos^2 k_j)}{d^2-\left(\sum_{j=1}^d\cos k_j\right)^2} \label{k1} 
\end{align}
From the Cauchy-Schwartz inequality, we have 
\begin{equation}\label{k2}
\left(\sum_{j=1}^d\cos k_j\right)^2 \leq d\sum_{j=1}^d\cos^2 k_j.
\end{equation}
Combining Eq.~(\ref{k1}) with Eq.~(\ref{k2}), we obtain 
\[ \sum_{j=1}^d x_j^2 \leq  1/d. \]
In particular, $\sum_{j=1}^d x_j^2 =  1/d$ if and only if $c=\cos k_j$ for all $j\in\{1,\dots,d\}$. 
\end{proof}
\noindent 
We now investigate what happens on the boundary of sphere $B_d$, defined by $\partial B_d$. 
\begin{theorem}\label{singular}
Define $V_{d}$ as the set of vertices of the $d$-dimensional cube inscribed in $B_d$, 
i.e., $V_{d}=\{\mathbf{x}=(x_1,\dots,x_d)\in \mathbb{Z}^d: |x_j|=1/d\}$. 
Then  
\begin{equation*}
\rho_d(\mathbf{x})= 
\begin{cases}
	\infty & \text{: $\mathbf{x}\in V_{d}$,} \\
        0 & \text{: $\mathbf{x}\in \partial B_d\smallsetminus V_{d}$ with $d\geq 3$,} \\
        \frac{4}{\pi^2 \cos^2 2\gamma} & \text{: $\mathbf{x}\in \partial B_d\smallsetminus V_{d}$ with $d=2$,}
\end{cases}
\end{equation*}
where we put $x=1/\sqrt{2}\cdot\cos\gamma$ and $y=1/\sqrt{2}\cdot\sin\gamma$ for $d=2$ case. 
\end{theorem}
\begin{proof}
For $\mathbf{k}=(k_1,\dots,k_d)$, let 
\begin{equation}\label{ishinomaki} 
f_j(\mathbf{k})=\frac{1-\cos^2k_j}{d^2-(\cos k_1+\cdots+\cos k_d)^2}(=x_j^2). 
\end{equation}
From Eq.~(\ref{k2}), we have 
\begin{equation}\label{d2'} 
f_j(\mathbf{k})\leq \frac{1}{d}. 
\end{equation} 
\segawa{We are interested in the solutions for which Eq.~(\ref{d2'})} becomes an equality. 
\segawa{Equality holds if and only if $\mathbf{x}=(x_1,\dots,x_d)\in \partial B_d$; that is, $c_j\equiv \cos k_j=c$ for some $c\in [-1,1]$.} 
Recall that Lemma~\ref{detJ} states 
\begin{equation}\label{amachan} 
|\mathrm{det}(\mathrm{H}_d)|=\prod_{l=1}^d\left(\frac{\cos k_l}{d\sin\phi(\mathbf{k})}\right)\left( 1-d\cos\phi(\mathbf{k}) \sum_{j=1}^d \frac{x_j^2}{\cos k_j}\right). 
\end{equation}
\HS{Since the product in the RHS becomes infinity when $|c|=1$, we have to treat this case carefully. }
\begin{enumerate}
\item $|c|<1$ case: \\
In this case, $\sin \phi(\mathbf{k})=\sqrt{1-c^2}\neq 0$ holds.  
Since $\mathbf{x}\in \partial B_d$, $\sum_{j=1}^d x_j^2=1/d$.
Inserting $\cos k_j=c$ for all $j\in \{1,\dots,d\}$ into Eq.~(\ref{amachan}), 
we get $\mathrm{det}(\mathrm{H}_d)=0$. 
Moreover, inserting $\cos k_j=c$ for all $j\in \{1,\dots,d\}$ into Eq.~(\ref{ishinomaki}), 
we can observe $f_j(\mathbf{k})=x^2_j=1/d^2$, implying that $(x_1,\dots, x_d)\in V_d$. 
Therefore, from Eq.~(\ref{sanriku}), $\rho_d(\mathrm{x})$ is infinity for $\mathrm{x}\in V_d$. 
\item $|c|=1$ case: \\
First let $c=1$. 
We take the limit of $(k_1,\dots,k_d)\in \mathbb{R}^d$ to $(0,\dots,0)$. 
To this end, we put $k_j=k_j(\epsilon)$ with $\lim_{\epsilon\to 0}k_j(\epsilon)=0$, 
$K_2(\epsilon)=\sum_{j=1}^d k_j^2(\epsilon)$ and $K_4(\epsilon)=\sum_{j=1}^d k_j^4(\epsilon)$. 
\HS{For any sufficiently small $\epsilon$, we obtain }
\begin{align*} 
\cos k_j &\sim 1-\frac{1}{2}k^2_j(\epsilon)\left( 1-\frac{1}{12}k_j^2(\epsilon) \right),\; 
\sin k_j\sim k_j(\epsilon)\left(1-\frac{1}{6}k_j^2(\epsilon)\right), \\
\cos \phi &\sim 1-\frac{K_2(\epsilon)}{2d}\left(1-\frac{1}{12}\frac{K_4(\epsilon)}{K_2(\epsilon)}\right),\;
\sin^2 \phi \sim \frac{K_2(\epsilon)}{d}\left(1-\frac{1}{12}\frac{K_4(\epsilon)}{K_2(\epsilon)}-\frac{1}{4d}K_2(\epsilon)\right).
\end{align*}
Inserting these approximations into Eq.~(\ref{ishinomaki}), we have 
\begin{equation}\label{x2}
x_j^2\sim \frac{k_j^2(\epsilon)}{d\; K_2(\epsilon)}
	\left( 1-\frac{k_j^2(\epsilon)}{3}+\frac{1}{12}\frac{K_4(\epsilon)}{K_2(\epsilon)}+\frac{1}{4d}K_2(\epsilon) \right).
\end{equation}
From Eq.~(\ref{x2}), we have 
\begin{equation}\label{sum}
1-d\cos\phi(\mathbf{k}) \sum_{j=1}^d \frac{x_j^2}{\cos k_j}
	\sim -\frac{1}{4}\frac{K_4(\epsilon)}{K_2(\epsilon)}+\frac{1}{4\;d}K_2(\epsilon)
\end{equation}
We also have 
\begin{equation}\label{prod}
\prod_{j=1}^d \frac{c_j}{d\sin (\phi(\mathbf{k}))}\sim \frac{1}{(d\;K_2(\epsilon))^{d/2}}. 
\end{equation}
Combining Eq.~(\ref{detJ1}) with Eqs.~(\ref{sum}) and (\ref{prod}), we obtain 
\begin{equation}\label{nichiyou}
\mathrm{det}(\mathrm{H}_d)\sim \frac{1}{(d\;K_2(\epsilon))^{d/2}} \frac{ K_2(\epsilon)}{4d}\left( 1-d\;\frac{K_4(\epsilon)}{K_2^2(\epsilon)} \right).
\end{equation}
From the Cauchy-Schwarz inequality, we find that  
\[ \frac{K_4(\epsilon)}{K_2^2(\epsilon)}=\frac{\sum_{j=1}^d k_j^4(\epsilon)} {\left(\sum_{j=1}^d k_j^2(\epsilon)\right)^2}
	\leq \frac{\sum_{j=1}^d k_j^4(\epsilon)}{d\;\sum_{j=1}^d k_j^4(\epsilon)}=\frac{1}{d}. \]
Note that the equality holds 
	\begin{center}$\Leftrightarrow$ $k_j(\epsilon)=k(\epsilon)$ for every $j$ \\ $\Leftrightarrow$ $\lim_{\epsilon \to 0}(x_1,\dots,x_d)\in V_d$. \end{center}
Moreover, from Eq.~(\ref{nichiyou}), the equality holds if and only if $\mathrm{det}(\mathrm{H}_d)= 0$ for every $\epsilon$. 
On the other hand, if the inequality holds, we have 
\begin{equation}\label{same}
|\mathrm{det}(\mathrm{H}_d)|=
	\begin{cases}
        \infty & \text{: $d\geq 3$, } \\
        \frac{1}{16}(1-2\kappa) & \text{: $d=2$, }
        \end{cases}
\end{equation}
where \sato{$\kappa=\lim_{\epsilon\to 0}K_4(\epsilon)/K_2^2(\epsilon)$}. 
Similarly, for $c=-1$ case, specifying $k_j=\pi+k_j(\epsilon)$ and taking limit $\epsilon\to 0$, we also obtain Eq.~(\ref{same}). 
When $d=2$, from Eq.~(\ref{x2}), we can write $x_1$ and $x_2$ in terms of a parameter $\gamma\mathbb{R}$
\[ x_1^2=\frac{1}{2}\cos^2\gamma,\;x_2^2=\frac{1}{2}\sin^2\gamma \]
in the limit of $\epsilon\to 0$. 
Inserting $\kappa=\cos^4\gamma+\sin^4\gamma$ into Eq.~(\ref{same}), 
\[ \frac{1}{4\pi^2}|\mathrm{det}(\mathrm{H}_2)|^{-1}=\frac{4}{\pi^2\cos^2 2\gamma}. \]
Indeed in Eq.~(\ref{lim2}), putting $x=1/\sqrt{2}\cdot \cos\gamma$, $y=1/\sqrt{2}\cdot \sin\gamma$, 
we have 
\[ \rho_2(x,y)=\frac{4}{\pi^2 \cos^2 2\gamma}. \]
It is completed the proof.
\end{enumerate}
\end{proof}

\par\noindent
\noindent
{\bf Acknowledgments.}
We thank the anonymous referee for valuable comments. 
YuH's work was supported in part by JSPS Grant-in-Aid for Scientific Research (C) 20540113, 25400208 and (B) 24340031.
NK and IS also acknowledge financial supports of the Grant-in-Aid for Scientific Research (C) from Japan Society for the Promotion of Science (Grant No. 24540116 and No. 23540176, respectively). 
ES thanks to the financial support of the Grant-in-Aid for Young Scientists (B) of Japan Society for the Promotion of Science (Grant No.25800088).
\par
\
\par

\begin{small}
\bibliographystyle{jplain}

\end{small}


\end{document}